\documentclass[aos,preprint]{imsart}
\RequirePackage[OT1]{fontenc}
\RequirePackage{amsthm,amsmath}
\RequirePackage[numbers]{natbib}
\RequirePackage[colorlinks,citecolor=blue,urlcolor=blue]{hyperref}
\usepackage{verbatim}
\usepackage{graphicx}
\usepackage{multirow}
\usepackage[toc,page]{appendix}
\usepackage{xcolor}
\usepackage[font=small,labelfont=bf]{caption}
\usepackage{graphicx}
\usepackage{relsize}
\usepackage{multicol}
\usepackage{lipsum}
\usepackage{mwe}
\usepackage{microtype}
\usepackage{subfigure}
\usepackage{booktabs} % for professional tables
\usepackage{amsfonts,amsbsy,amssymb,amsthm,amsmath}
\newtheorem{corol}{Corollary}[section]
\newtheorem{definition}{Definition}[section]  
    
\newtheorem{lemma}{Lemma}[section]
\newtheorem{theorem}{Theorem}[section]
\newcommand\mbB{{\mathbb B}}
\newcommand\mbD{{\mathbb D}}
\newcommand\mcD{{\mathcal D}}
\newcommand\mcF{{\mathcal F}}
\newcommand\mbH{{\mathbb H}}
\newcommand\mcH{{\mathcal H}}
\newcommand\bK{{\bf K}}
\newcommand\mbK{{\mathbb K}}
\newcommand\mcK{{\mathcal K}}
\newcommand\bM{{\bf M}}
\newcommand\bm{{\bf m}}
\newcommand\mcM{{\mathcal M}}
\newcommand\mcN{{\mathcal N}}
\newcommand\mbR{{\mathbb R}}
\newcommand\mcS{{\mathcal S}}

\newcommand\bY{{\bf Y}}
\newcommand\mcY{{\mathcal Y}}

\newcommand\bvep{\boldmath{\varepsilon}}
\newcommand\bmu{\boldmath{\mu}}

\DeclareMathOperator*\argmin{argmin}
\DeclareMathOperator\Cov{Cov}
\DeclareMathOperator\E{E}

\DeclareMathOperator\Span{span}

\usepackage[margin=1in]{geometry}

\begin{document}

\begin{frontmatter}
\title{Formal Privacy for Functional Data \\ with Gaussian Perturbations}
\runtitle{Privacy for FDA}

\begin{aug}
		\author{ \fnms{Ardalan} \snm{Mirshani} %\thanksref{t1}
			\ead[label=e1]{azm245@psu.edu}
			\ead[label=u1,url]{www.ArdalanMirshani.com}}
		\author{ \fnms{Matthew} \snm{Reimherr}\thanksref{t1}
			\ead[label=e2]{mreimherr@psu.edu}
			\ead[label=u2,url]{www.personal.psu.edu/mlr36}}
		\and
		\author{ \fnms{Aleksandra} \snm{Slavkovic} \thanksref{t2}
			\ead[label=e3]{sesa@psu.edu}
			\ead[label=u3,url]{www.personal.psu.edu/abs12}}

		\thankstext{t1}{Suported by NSF DMS-1712826}
		\thankstext{t2}{Supported by  NSF SES-1534433, NIH UL1 002014}
		\runauthor{Mirshani, Reimherr, and Slavkovic}
		
		\affiliation{Pennsylvania State University}
		
		%\address{Department of Statistics \\
		%	Pennsylvania State University \\
		%	\printead{e1}\\
		%	\phantom{E-mail:\ }\printead*{e2} \\
		%	\phantom{E-mail:\ }\printead*{e3} \\
		%	\printead{u1} \\
		%	\phantom{URL:\ }\printead*{u2} \\
		%	\phantom{URL:\ }\printead*{u3}}		

	\end{aug}

%\twocolumn[
%\icmltitle{Formal Privacy for Functional Data with Gaussian Perturbations}
%\icmltitle{Formal Privacy for Functional Data with Applications to Penalized Smoothing}
%\icmltitle{Est Statistical Privacy for Functional Data Using Functional Densities}






\begin{abstract}
	Motivated by the rapid rise in statistical tools in {\it Functional Data Analysis}, we consider the Gaussian mechanism for achieving differential privacy with parameter estimates taking values in a, potentially infinite-dimensional, separable Banach space.  Using classic results from probability theory, we show how densities over function spaces can be utilized to achieve the desired differential privacy bounds.  This extends prior results of \citet{hall:rinaldo:wasserman:2013} to a much broader class of statistical estimates and summaries, including ``path level" summaries, nonlinear functionals, and full function releases.  By focusing on Banach spaces, we provide a deeper picture of the challenges for privacy with complex data, especially the role regularization plays in balancing utility and privacy.  Using an application to penalized smoothing, we explicitly highlight this balance in the context of mean function estimation.  Simulations and an  application to {diffusion tensor imaging} are briefly presented, with extensive additions included in a supplement.
\end{abstract}


\end{frontmatter}

 \section{Introduction}
 
 New studies, surveys, and technologies are resulting in ever richer and more informative data sets.  Data being collected as part of the ``big data revolution'' occurring in the sciences have dramatically expanded the pace of scientific progress over the last several decades, but often contain a significant amount of personal or subject level information. These data and their corresponding analyses present substantial challenges for preserving subjects' privacy as researchers attempt to understand what information can be publicly released without impeding scientific advancement and policy making \citep{lane2014privacy}.  %Data being collected as part of the ``big data revolution'' occurring in the sciences is especially challenging as these data often contain a tremendous amount of personal information.

 One type of big data that has been heavily researched in the statistics community over the last two decades is \textit{functional data}, with the corresponding branch of statistics called \textit{functional data analysis}, FDA.  FDA is concerned with conducting statistical inference on samples of functions, trajectories, surfaces, and other similar objects.  Such tools have become increasingly necessary as our data gathering technologies become more sophisticated.  FDA methods have proven very useful in a wide variety of fields including economics, finance, genetics, geoscience, anthropology, and kinesiology, to name only a few \citep{ramsay2002applied,ramsay:silverman:2005,ferraty:2011,horvath2012inference,kokoszka2017introduction}.  Indeed, nearly any data rich area of science will eventually come across applications that are amenable to FDA techniques. However, functional and other high dimensional data are also a rich source of potentially personally identifiable information 
 \citep{kulynych2002legal,erlich2014routes,schadt2012bayesian}.
 %(e.g., genomic data, including genomic sequences \citet{erlich2014routes,schadt2012bayesian}, biomedical imaging \citet{kulynych2002legal}, etc.). %\textcolor{red}{I DO REF the PNAS paper in the last section so should probably dig out some other refs here or bring that to the front}
 
{\bf Related Work:} To date, there has been very little work concerning FDA and statistical data privacy, in either \textit{Statistical Disclosure Limitation}, SDL or \textit{Differential Privacy}, DP. SDL is the branch of statistics concerned with limiting identifying information in released data and summaries while maintaining their utility for valid statistical inference, and has a rich history for both methodological developments and applications for ``safe" release of altered (or masked) microdata and tabular data \citep{Dalenius77,Rubin93,WillenborgD96,fienberg10, HundepoolD12}. DP has emerged from theoretical computer science with a goal of designing privacy mechanisms with mathematically provable disclosure risk \citep{dwork:2006,Dwork2006:Sensitivity}.  %, for a general review of the area and the related methodology. 
 \citet{hall:rinaldo:wasserman:2013} provide the most substantial contribution to statistical privacy with FDA to date, working within the DP framework and the Gaussian mechanism for releasing a finite number of point-wise evaluations, with applications to kernel density estimation and support vector machines.  They provide a limiting argument that establishes DP for certain sets of functions.  One of the major findings of \citet{hall:rinaldo:wasserman:2013} is the connection between DP and Reproducing Kernel Hilbert Spaces, which we extend more broadly to {\it Cameron-Martin Spaces}.  Recently, \citet{alda2017bernstein} extended the work of \citet{hall:rinaldo:wasserman:2013} by considering a Laplace (instead of a Gaussian) mechanism and focused on releasing an approximation based on Bernstein polynomials, exploiting their close connection to point-wise evaluation on a grid or mesh.  %Unfortunately, the noise they consider consists of independent Laplace perturbations which can result in a sizable loss in utility when a large number of evaluations are used.  
 %In contrast, \citet{hall:rinaldo:wasserman:2013} use a Gaussian process, which is relatively insensitive to the number of evaluations.   
 Other related contributions include \citet{alvim2018metric} who consider privacy over abstract metric spaces assuming one has a sanitized dataset, and \citet{smith2018differentially} who examine how to best tailor the mechanism from \citet{hall:rinaldo:wasserman:2013}. %I think this is a wrong ref here so I replaced it \citet{hall:2006}. %{\color{red}HOW's our work different from these two new references? We say this for HallRinWas below but need to say something wrt these new refs too! We could also comment more on this on page 6 when we talk about using Hall et al 2013 as the backbone for our mechanism, that is how does Alda's work fit here.}

 {\bf Our Contribution:} In this work, we move beyond \citet{hall:rinaldo:wasserman:2013,alda2017bernstein} and \citet{smith2018differentially} by establishing DP for functional data much more broadly.  We first show that the Gaussian mechanism achieves DP for a large class of linear functionals and then show that this mechanism offers seemingly complete protection against any summary imaginable, covering any ``path level" summaries, nonlinear transformations, or even a full function release, though the later is usually not computationally feasible without some additional structure (e.g., continuous time Markov chains).  Such extensions are critical for working with transformations that are not simple point wise evaluations, such as basis expansions, norms, and derivatives or when the objects exhibit complex nonlinear dynamics.  We also provide an interesting negative result, that shows that not all Gaussian noises are capable of achieving DP for a particular summary, regardless of how the noise is scaled.  In particular, we introduce a concept called \textit{compatibility}, and show that if a noise is not compatible with a particular summary, then it is impossible to achieve DP.
 To establish the necessary probabilistic bounds for DP we utilize functional densities via the {\it Cameron-Martin Theorem}.  This is also of independent interest in FDA as densities for functional data are rarely utilized due to the lack of a natural base measure \citep{berrendero2017use}.  Most attempts at utilizing or defining densities for functional data involve some work-around to avoid working in infinite dimensions \citep{delaigle:hall:2010,dai2017optimal}.  Lastly, we demonstrate these tools by considering mean function estimation via penalized smoothing, where in addition to the DP results we also provide guarantees on the utility of the sanitized estimate.
 
One of the major findings of this work is the interesting connection between \textit{regularization} and privacy. In particular, we show that by slightly over smoothing, one can achieve DP with substantially less noise, thus better preserving the utility of the release.  This is driven by the fact that a great deal of personal information can reside in the ``higher frequencies'' of a functional parameter estimate, while the ``lower frequencies'' are typically shared across subjects.  To more fully illustrate this point, we demonstrate how a cross-validation for choosing smoothing parameters can be dramatically improved when the cross-validation incorporates the function to be released.  % \textcolor{red}{[Since the value of CV and PCV represent two different things, I am wondering if we could compare them like this?]}.
 Previous works concerning DP and regularization have primarily focused on performing shrinkage regression in a DP manner \citep[e.g.][]{kifer12perm,Chaudhuri2011:DPERM} and model selection with linear regression (e.g., \citet{lei2018DPmodelselect}), not exploiting the regularization to recover some utility as we propose here. %LINKS BETWEEN REGularization and DP have been made before in the CS/ML literature, so I need to dig out some refs for this and we need to make sure to say what we bring more clearly, e.g., point how this is extended to the FDA context]
 
 {\bf Organization:} The remainder of the paper is organized as follows.  In Section 2 we present the necessary background material for DP and FDA.  In Section 3 we present our primary results concerning releasing a finite number of linear functionals followed by full function and nonlinear releases.  %We also present our work related to densities for FDA, which is of independent interest.  
 In Section 4 we present an application to penalized smoothing for mean estimation, which is especially amenable to our privacy mechanism.  %There we provide bounds on the global sensitivity as well as guarantees on the utility that highlight an interesting relationship between the utility, the noise to be added for privacy, and the RKHS used for smoothing.  
 In Section 5 we present simulations to highlight the role of different parameters, while in Section 6 we present an application of Diffusion Tensor Imaging of Multiple Sclerosis patients. % and another based on the 3D facial imaging.  
 In Section 7 we discuss our results and present concluding remarks.  %All mathematical proofs and derivations can be found in the Appendix.

 % Other papers
 % holohan, leith, Mason, 2015
 % smith, alvarez, zwiessele, lawrence, (arxiv)
 % alda, rubinstein, (arxiv)

 %%%%%%%%%%%
 \section{Background}
 \subsection{Differential Privacy}
 {\it Differential Privacy}, DP, was first introduced in \citet{dwork:2006,Dwork2006:Sensitivity}. %, as a mathematical framework for privacy. 
 Let $\mbD$ be a (potentially infinite) population of records, and denote by $\mcD$ %the power set of $\mbD$.
 the collection of all $n$-dimensional subsets of $\mbD$.
 Throughout we let $D$ and $D^\prime$ denote elements of $ \mcD$.  Notationally, we omit the dependence on $n$ for ease of exposition.  % but will remind the reader when the assumptions of each theorem are stated.  %In general, our goal is to release a private summary of $D$ which, in the present work, we assume is a function.  %We denote this summary as $f : \mcD \to \mbH$, where $\mbH$ is a real separable Hilbert space with inner product $\langle \cdot, \cdot \rangle_{\mbH}$.
We work with  $(\epsilon,\delta)$-DP, where $\epsilon$ and $\delta$ could be considered as parameters representing the {\it privacy budget} with smaller values indicating stronger privacy; when $\delta=0$ we obtain pure- or $\epsilon$-DP. DP is a property of the privacy mechanism applied to the data summary, in this case $f_D:=f(D)$, prior to release.  For simplicity, we will denote the application of this mechanism using a tilde; so $\tilde f_D:= \tilde f(D)$ is the {\it sanitized} version of $f_D$.  Probabilistically, we view $\tilde f_D$ as a random variable indexed by $D$ (which is not treated as random).
 This criteria can be defined for any probability space.
 \begin{definition}[\citet{dwork:2006,wasserman:zhou:jasa09}] \label{d:dp}
 	Let $f:\mcD \to \Omega$, where $( \Omega, \mcF)$ is some measurable space.  Let $\tilde f_D$ be random variables, indexed by $D$, taking values in $\Omega$ and representing the privacy mechanism.
 	The privacy mechanism is said to achieve $(\epsilon, \delta)-$DP if for any two datasets, $D$ and $D'$, which differ in only one record, we have
 	\[
 	P(\tilde f_D \in A) \leq P(\tilde f_{D'} \in A) e^\epsilon + \delta,
 	\]
 	for any measurable set $A \in \mcF$.
 \end{definition}
 The most common setting is when $\Omega = \mbR$, i.e., the summary is real valued.  In Section \ref{s:finite} we take $\Omega=\mbR^N$, corresponding to releasing $N$ linear functionals of a functional object, while in \ref{s:function} we will consider the case when $\Omega = \mbB$, a real separable Banach space.  In \citet{hall:rinaldo:wasserman:2013}, they were able to consider the space of real valued functions over $\mbR^d$, i.e., the product space $\Omega=\mbR^T$ with $T = \mbR^d$ (or some compact subset), by clever limiting arguments of cylindrical sets; they thus considered DP over $\mbR^T$ equipped with the cylindrical $\sigma$-algebra (i.e. the smallest $\sigma$-algebra that makes point-wise evaluations measurable). However, in most cases we are actually interested in a subspace of $\mbR^T$, such as the space of continuous functions, square integrable functions, differentiable functions, etc.  It turns out that the resulting $\sigma$-algebras (and thus the protection offered by DP) are in general quite different, and that working directly with $\mbR^T$ can result is some fairly glaring holes.  Chapter 7 of \citet{billingsley2008probability} or Section 3.1 of \citet{bogachev1998gaussian} discuss these issues, but it is interesting to note that the cylindrical $\sigma$-algebra on $\mbR^T$ is missing the sets of linear functions, polynomials, constants, nondecreasing functions, functions of bounded variation, differentiable functions, analytic functions, continuous functions, functions continuous at a given point, and Borel measurable functions.  To avoid this issue, we work directly with the Borel $\sigma$-algebra on the function space of interest, which in our case will always be a Banach space, though in principle this approach can be extended to handle any locally convex vector space. 
 
 %As such, in  their case $\mcF$ was taken as the cylindrical $\sigma$-algebra.  Interestingly, in infinite dimensional spaces, the cylindrical $\sigma$-algebra can be strictly smaller than the Borel $\sigma$-algebra.  In extreme cases, the cylindirical $\sigma$-algebra may not even contain the continuous functions (CITE Billingsley).  To avoid this issue, we take a different approach, which relies on a powerful tool from stochastic processes known as the {\it Cameron-Martin Theorem}.
 %
 
 At a high level, achieving  $(\epsilon,\delta)-$DP means that the object to be released changes relatively little if the sample on which it is based is perturbed slightly. This change is related to what \citet{dwork:2006} called {\it sensitivity}.  Another nice feature is that if $\tilde f_D$ achieves DP, then so does any measurable transformation of it; see \citet{Dwork:2006b,Dwork2006:Sensitivity} for the original results,  \citet{wasserman:zhou:jasa09} for its statistical framework, and \citet{Dwork2014:AFD} for a more recent detailed review of relevant DP results. %For convenience, we restate the lemma and its proof with respect to our notation in Appendix \ref{s:proof}.

 \subsection{Functional Data Analysis}
 Much of FDA is built upon the {\it Hilbert space} approach to modeling, i.e., viewing data and/or parameters as elements of a particular Hilbert space (most commonly $L^2[0,1]$ after possibly rescaling). However, we take a more general approach by allowing for arbitrary separable Banach spaces, which will dramatically increase the application of our results, while requiring only a small amount of more technical work.  All of the concepts/tools from this section are classic probability theory results that might be of interest in the FDA and privacy communities.  We refer the interested reader to \citet{bogachev1998gaussian} for a nearly definitive treatment of Gaussian measures.  Throughout we let $\mbB$ denote a real separable Banach space; we always implicitly assume that $\mbB$ is equipped with its Borel $\sigma$-algebra. %The most commonly encountered space in FDA is $L^2[0,1]$, though many other spaces are becoming of increasing interest.  For example, multivariate or panel functional data, functional data over multidimensional domains (such as space-time), as well as spaces that incorporate more structure on the functions, such as Sobolev spaces and Reproducing Kernel Hilbert Spaces.
 
 Let $\theta: \mcD \to \mbB$ denote the particular summary of interest and for notational ease, we define $\theta_D := \theta(D)$.  In Section \ref{s:finite} we consider the case where the aim is to release a finite number of linear functionals of $\theta_D$, whereas in Section \ref{s:function} we consider releasing sanitized versions of the entire function or some nonlinear transformation of it (such as a norm or basis expansion).

 The backbone of our privacy mechanism is the same as in \citet{hall:rinaldo:wasserman:2013}, and is used extensively across the DP literature.  In particular, we add Gaussian noise to the summary and show how the noise can be calibrated to achieve DP.
 A random process $X \in \mbB$ is called {\it Gaussian} if $f(X)$ is Gaussian in $\mbR$, for any continuous linear functional $f \in \mbB^*$ \citep[Def. 2.2.1]{bogachev1998gaussian}.  Throughout we use $*$ to denote the corresponding topological dual space.  Equipped with the norm $\|f\|_{\mbB^*} = \sup_{\|h\|_{\mbB} \leq 1} f(h)$, the dual space is also a separable Banach space.    %Equivalently, a probability measure $\nu$ over $\mbB$ is called a Gaussian measure if the push forward measure $\nu \circ f^{-1}$ on $\mbR$ is Gaussian for any $f \in \mbB^*$.  
 The pair $(\mbB,\nu)$ is often called an {\it abstract Weiner space} \citep[Sec. 3.9]{bogachev1998gaussian}, where $\nu$ is the probability measure over $\mbB$ induced by $X$.  Every Gaussian process is uniquely parametrized by a mean, $\mu \in \mbB$, and a covariance operator $C : \mbB^* \to \mbB$, which for every $f \in B^*$ satisfies
 \[
 \E[f(X)] = f(\mu), \qquad C(f) = \E[f(X-\mu) (X-\mu)]
 \]
  \cite{laha:rohatgi:1979}.  One can equivalently identify $C$ as a bilinear form $C(f,g) = \Cov(f(X),g(X))$, and we will use both notations whenever convenient.  
It follows that  
 \[
 f(X) \sim \mcN(f(\mu), C(f,f)),
 \]
 for any $f \in \mbB^*$.  We use the short hand notation  $\mcN$ to denote the Gaussian distribution over $\mbR$, but include subscripts for any other space, e.g.,\ $\mcN_{\mbB}$ for  $\mbB$.

A key object concerning privacy will be the {\it Cameron-Martin space} \citep[Sec. 2.4]{bogachev1998gaussian} of $X$ (or equivalently of $(\mbB,\nu)$).  Using $C$ one can equip $\mbB^*$ with an inner product
\begin{align*}
 \langle f, g \rangle_{\mcK} :=&  \Cov(f(X), g(X))  \\
=& \int f(x - \mu) g(x-\mu) \ d \nu(x),
%\leq \Trace(C) \|f\|_{\mbB^*} \|g\|_{\mbB^*},
\end{align*}
which implies that $\mbB^*$ can be identified as a linear subspace of $L^2(\mbB ,\nu)$, that is, $\nu$-square integrable functions over $\mbB$.  
However, $\mbB^*$ is no longer complete under this inner product; denote the completed space as $\mcK$.  %In particular, let $e_i$ be a basis for $\mbB^*$ that is also orthonormal in $\mcK$.  Then every element in $\mcK$ will have the form
%\[\sum a_i e_i \qquad \text{with} \qquad \sum a_i^2< \infty,\] 
%but notice that this sum need not converge in $\mbB^*$.  
 Finally, consider the set of all $h \in \mcH \subset \mbB$ such that the mapping
\[
f \to f(h) % \Longrightarrow T_h : \mcK \to \mbR,
\]  
is continuous in the $\mcK$ topology. Intuitively, these functions are ones that are "nicer" than arbitrary elements of $\mbB$.  In particular, they must be regular enough to ensure that $f(h)$ is finite for any $f \in \mcK$, which are much "uglier" functionals than those in $\mbB^*$.  By the Riesz representation theorem, we can associate each element $h \in \mcH$ with a $T_h \in \mcK$ such that $\langle T_h, f \rangle_\mcK = f(h)$.  The set $\mcH$ equipped with the inner product
\[
\langle x, y \rangle_{\mcH} = \langle T_x, T_y \rangle_{\mcK}, 
\]
is called the {\it Cameron-Martin Space}, and is itself a separable Hilbert space.  Note that, slightly less abstractly, we have $C(T_h) = h$ \citep[Lemma 2.4.1]{bogachev1998gaussian}.  One can also view $\mcK$ as being a type of {\it Reproducing Kernel Hilbert Space} \citep[pg. 44]{bogachev1998gaussian} in a very broad sense since we have $\langle T_h, f\rangle_{\mcK} = f(h)$, for any $f \in \mcK$.  In infinite dimensions the Cameron-Martin space does not contain the sample paths of $X$, but they can be thought of as "living at the boundary" of $\mcH$.  While the Cameron-Martin space is introduced via Gaussian processes, it is determined entirely by the covariance operator $C$.  %Another fascinating connection concerns Reproducing Kernel Hilbert Spaces as Cameron-Martin spaces posses a type of reproducing property that holds in a very broad sense.  In particular, for any $f \in \mbB^\star$ one has, by construction, 
% $\langle T_x, f \rangle_{\mcK} = f(x).$
%If $x$ was its self a function over some domain, and point-wise evaluations were continuous, then this would imply the more classic reproducing property of RKHS.

\subsection{Hilbert Space Example}
While working with a general Banach space allows for a broader impact, it is also conceptually much more challenging.  We can gain additional insight by considering what happens when $\mbB = \mbH$ is actually a Hilbert space.  By the Riesz Representation Theorem, $\mbH$ is isomorphic to $\mbH^*$ so we can always identify $\mbH^*$ with $\mbH$ and de-emphasize the linear functionals.

We can obtain very convenient expressions if we take a basis $\{v_i: i=1,2,\dots\}$ of $\mbH$ consisting of the eigenfunctions of $C$ (recall in Hilbert spaces $C$ must be nonnegative definite and trace class).  In this case we have that
\[
C(v_i) = \lambda_i v_i. 
\]
So assuming that that there are no zero eigenvalues, define $e_i = \lambda_i^{-1/2} v_i$, then these form an orthonormal basis of $\mcK$ as
\[
\langle e_i, e_j \rangle_\mcK
= \lambda_i^{-1/2} \lambda_j^{-1/2} \Cov(\langle v_i,X \rangle_{\mbH}, \langle v_j, X\rangle_\mbH) = \delta_{ij},
\]
where $\delta_{ij}$ is 1 if $i=j$ and zero otherwise.  The space $\mcK$ consists of all linear combinations of the $e_i$ whose coefficients are square summable.  The inner product on Cameron-Martin space, $\mcH$, is given by
\[
\langle x ,y \rangle_{\mcH}
%= \sum \langle x,e_i \rangle_{\mcK} \langle x, e_i \rangle_\mcK
= \sum \frac{\langle x, v_i \rangle_{\mbH} \langle y, v_i \rangle_{\mbH}}{\lambda_i},
\] 
so that
\begin{align}
\mcH := \left\{ x \in \mbH : \sum_{i=1}^\infty \frac{\langle x, v_i \rangle^2_{\mbH}}{\lambda_i} < \infty \right\}. \label{e:K}
\end{align}
In other words, those elements of $\mcH$ are the functions whose coefficients in the $v_i$ basis decrease sufficiently quickly.  Note that the case where some $\lambda_i$ are actually zero (meaning $C$ has a nontrivial null space) can be easily handled by restricting $\mcH$ to the range of $C$.\footnote{In fact, such a game can be played quite broadly as any Radon measure over a {\it Fr\'echet space} will concentrate on a reflexive separable Banach space \citep[Thm 3.6.5]{bogachev1998gaussian}.}

 The space $\mcH$ is a Hilbert space when equipped with the inner product $\langle x, y \rangle_{\mcH} = \sum \lambda_i^{-1} \langle x, v_i \rangle \langle y, v_i \rangle$.  When $\mbH = L^2[0,1]$ and $K$ is an integral operator with continuous kernel $k(t,s)$, then $\mcH$ is isomorphic to a Reproducing Kernel Hilbert Space, RKHS \citep{berlinet:thomas:2011} (one has to be slightly careful as $L^2$ consists of equivalence classes).

 %%%%%%%%%%%%%%%%%%%%
 \section{Privacy Enhanced Functional Data}
 In this section we present our main results.
 The mechanism we use for guaranteeing differential privacy is to add a Gaussian noise before releasing $\theta_D$.  In other words, our releases will be based on a private version $\tilde \theta_D$, where  $\tilde \theta_D = \theta	_D + \sigma Z$.  However, it turns out that not just any Gaussian noise, $Z$, can be used. In particular, the options for choosing $Z$ depend heavily on the summary $\theta_D$.  This is made explicit in the following definition.
 
 \begin{definition} \label{d:compatible}
 	We say that the summary $\theta_D$ is compatible with a Gaussian noise, $Z \sim \mcN_\mbB(0,C)$, if $\theta_D$ resides in the Cameron-Martin space of $Z$ for any $D \in \mcD$.
 \end{definition}
 
 Intuitively, this means that the noise must be ``rougher" than the summaries.  Our next definition is a generalization of one from  \citet{hall:rinaldo:wasserman:2013}, which focused on functions in RKHS only.
 \begin{definition}\label{d:gs}
 	The global sensitivity of a summary $\theta_D$, with respect to a Gaussian noise $Z\sim\mcN_\mbB(0,C)$ is given by
 	\[
 	\Delta^2 = \sup_{D^\prime \sim D} \| \theta_D - \theta_{D'}\|_{\mcH}^2,
 	\]
 	where $D^\prime \sim D$ means the two sets differ at one record only, and $\| \cdot \|_\mcH$ is the norm on the  Cameron-Martin space of $Z$.
 \end{definition}
 
 The {\em global sensitivity (GS)} is a central quantity in the theory of DP.  In particular, the amount of noise, $\sigma Z$, depends directly on the global sensitivity. There are other notions of ``sensitivity'' in the DP literature, such as local sensitivity, but here our focus is on the global sensitivity that typically leads to the {\it worst case} definition of risk under DP; for a detailed review of DP theory and concepts, see \citet{Dwork2014:AFD}. If a summary is not compatible with a noise, then it is possible to make the global sensitivity infinite, in which case no finite amount of noise would be able to preserve privacy in the sense of satisfying DP.
 %
 %Definition \ref{d:gs} was introduced in \citet{hall:rinaldo:wasserman:2013} while Definition \ref{d:compatible} we introduce here to emphasize that, for a particular problem, not all noise mechanisms can be used to maintain privacy.  This is in stark contrast with scalar and multivariate settings where the covariance structure of the noise plays a relatively unimportant role.
 %
 %If $\theta_D$ is not compatible with $Z$, then one cannot guarantee that a $\tilde f_D$ can be constructed which is private, and thus a different noise must be used.
 
 %{\color{red} DO WE PROPERLY HIGHLIGHT this THRM in our contributions/findings?}
 \begin{theorem}\label{t:comp}
 	If a summary $\theta_D \in \mbB$ is not compatible with a noise $Z\sim N_\mbB(0,C)$ then for any $\sigma > 0$, $\tilde \theta_D:= \theta_D + \sigma Z$ {\bf will not} satisfy DP.
 \end{theorem}
 \begin{proof}
 	%If the summary is not compatible, then this implies that there exists a $D$ such that $\| \theta_D\|_\mbK = \infty$.  From \citet{rao:1963} the distributions $\delta Z$ and $\theta_D + \delta Z$ are thus orthogonal for any $\delta > 0$ and so one cannot achieve DP in $\mbB$.
 	This is a direct consequence of the Cameron-Martin Theorem, which characterizes the equivalence/orthogonality of Gaussian measures.  If the summary is not compatible with the noise, then there exists a $D \sim D'$ such that $\| \theta_D - \theta_{D'}\|_\mcH = \infty$, which implies that the distributions $\tilde f(D)$ and $\tilde f(D')$ are singular. 
 \end{proof}
 
 Intuitively, if the summary is not compatible with the noise, then one can pool even small amounts of information from across an infinite number of dimensions to produce a disclosure. 
 
 %%%%%%%%%%%%%%%%%%
 \subsection{Releasing Finite Projections} \label{s:finite}
 We begin with the comparatively simpler task of releasing a finite vector of linear functionals of $\theta_D$.  In particular, we aim to release $f(\theta_D) = \{f_1(\theta),\dots,f_N(\theta_D)\}$, for $f_i \in \mcK \supset \mbB^*$ and some fixed $N$.  While placed in a more general context, the core concepts involved are the same as in \citet{hall:rinaldo:wasserman:2013} (they focused on point-wise evaluations, which are continuous linear functionals over an appropriate space).  Interestingly, since we are using the Cameron-Martin space, we can actually release more than just continuous linear functionals; we can release any functional from $\mcK$, which is, in general, much larger than $\mbB^*$.  
 \begin{theorem} \label{t:finite}
 	Assume $\theta_D$ is compatible with $Z \sim \mcN(0,C)$, $\epsilon \leq 1$, and define
 	\[
 	\tilde \theta_D = \theta_D + \sigma Z
 	\qquad \text{with} \qquad
 	\sigma^2 \geq  \frac{2 \log(2/\delta)}{\epsilon^2} \Delta^2.
 	\]  
 	Now define $f(\theta_D) = \{f_1(\theta),\dots,f_N(\theta_D)\}$ and $\tilde f(\theta_D) = \{f_1( \tilde \theta_D),\dots,f_N(\tilde \theta_D)\}$, for $\{f_i \in \mcK \}_{i=1}^{N}$.  Then $\tilde f_D$ achieves $(\epsilon,\delta)$-DP in $\mbR^N$. %\textcolor{red}{$[\{h_i \in \mbH\}_{i=1}^{N}]$}.
 \end{theorem}
 
 Theorem \ref{t:finite} can be viewed as an extension of \citet{hall:rinaldo:wasserman:2013} who focus on point-wise releases.  If $\mbB$ is taken to be the space of continuous functions with an appropriate topology, then Theorem \ref{t:finite} implies point-wise releases are protected as well.  However, this theorem allows the release of any functional in $\mcK$. This dramatically increases the release options and applications as compared to \citet{hall:rinaldo:wasserman:2013} or \citet{alda2017bernstein}.

 %%%%%%%%%%%%%%%%%%%%%%
 \subsection{Full Function and Nonlinear Releases} \label{s:function}
 While Section \ref{s:finite} covers a number of important cases, it does not cover all releases of potential interest.  In particular, full function releases are not protected and neither are nonlinear releases, such as norms or derivatives. \textcolor{black}{A full function release is not usually practically possible, however in some situations, such as continuous time Markov chains, full paths can be completely summarized using a finite number of values, but these values are not simple point wise evaluations or linear projections and thus not covered under \citet{hall:2006,alda2017bernstein} or our results from Section \ref{s:finite}.}
 Regardless of this point, there is a certain comfort in knowing that one has a complete protection that holds regardless of whatever special structures one might be able to exploit or new computational tools that might become available.  In addition, one can obtain a great deal of insight by considering the infinite dimensional problem, as it highlights the fundamental role smoothing plays when trying to maintain utility while achieving DP.
  %{\color{red} SHOULD WE NOTE HERE SOMETHING AGAIN shortcomings of Alda and Smith ...?}
  
 To guarantee privacy for these types of releases, we need to establish $(\epsilon,\delta)$-DP for the entire function, not just finite projections.  This means that in Definition \ref{d:dp}, the space is taken to be $\mbB$, which is infinite dimensional. %In \citep{hall:rinaldo:wasserman:2013}, they use a limiting argument to establish DP over sets contained in the cylindrical $\sigma$-algebra  for functions in $\mbR^T$ for $T = \mbR^d$ or some compact subset therein.  However, as pointed out in \citet{billingsley2008probability}, this class is surprisingly restrictive. In contrast, we establish over the entire Borel $\sigma$-sigma algebra.  
 In previous works \citep[e.g.,][]{dwork2014algorithmic,hall:rinaldo:wasserman:2013} to establish the probability inequalities like in Definition \ref{d:dp}, bounds based on multivariate normal densities were used.  This presents a serious problem for FDA and infinite dimensional spaces as it becomes difficult to work with such objects (there is very little FDA literature that does so).  For example,
 \citet{delaigle:hall:2010} define densities only for finite ``directions" of functional objects.  Another example is \citet{bongiorno:aldo:2015} who define psuedo-densities by carefully controlling ``small ball" probabilities.  Both works claim that for a functional object the density ``generally does not exist."  However, this turns out to be a \textit{technically} incorrect claim, while still often being true in spirit.  The correct statement is that, in general, it is difficult to define a \textit{useful} density for functional data.  In particular, to work with likelihood methods, a family of probability measures should all have a density with respect to the same base measure, which, at present, does not appear to be possible in general for functional data.

 \textcolor{black}{The difficulty in defining densities in infinite-dimensional spaces comes from the fact there is no common {\it base} or {\it reference} measure \citep{cuevas2014partial}, such as Lebesgue measure, however} our goal in using densities is more straightforward.
 \textcolor{black}{We require densities} (with respect to the same base measure) for the family of probability measures induced by $\{ \theta_D + \sigma Z: D \in \mcD\}$, where $Z$ is a mean zero Gaussian process in $\mbB$ with covariance operator $C$.  It turns out that this is in fact possible because we are adding the exact same type of noise to each element.  We give the following lemma, which is simply an application of the classic Cameron-Martin Theorem.
 %[DO WE NEED TO GIVE A PAGE AND OR THR NUMBER?]
 \begin{lemma}\label{t:density}
 	Assume that the summary $\theta_D$ is compatible with a noise $Z$.  Denote by $Q$ the probability measure over $\mbB$ induced by $ \sigma Z$, and by $\{P_D: D \in \mcD\}$ the family of probability measures over $\mbB$ induced by $\theta_D + \sigma Z$.  Then every measure $P_D$ has a density over $\mbB$ with respect to $Q$, which is given by
 	\[
 	\frac{dP_D}{dQ} (x) =
 	\exp\left\{
 	- \frac{1}{2\sigma^2}\left(
 	\| \theta_D\|_\mcH^2 - 2 T_{\theta_D}(x) 
 	\right)
 	\right\},
 	\]
 	$Q$ almost everywhere.  Recall that $\theta_D = C(T_{\theta_D})$ and that the density is unique up to a set of $Q$ measure zero.
 \end{lemma}
 \textcolor{black}{ At this point we stress that the noise is chosen by the user; it is not a property of the data.  The primary hurdle for the user is ensuring that the summary is compatible with the selected noise.  As we will see in Section \ref{s:g.s.bounds}, one can accomplish this by using specific estimators.}
 Lemma \ref{t:density} implies that, for any Borel measurable set $A \subset \mbB$ we have
 \[
 P_D(A) = \int_A \frac{dP_D}{dQ} (x) \ dQ(x),
 \]
 which we exploit in our proofs later on.  
 
 Now that we have a well defined density we can establish differential privacy for entire functions.
 \begin{theorem} \label{t:func}
 	Let $f_D: = \theta_D$, and assume that it is compatible with a noise $Z$ and that $\epsilon  \leq 1$, then $\tilde f_D := f_D + \sigma Z$ achieves $(\epsilon,\delta)$-DP over $\mbB$ (with the Borel $\sigma$-algebra), where $\sigma$ is defined in Theorem \ref{t:finite}.
 \end{theorem}
 We also have the following simple corollary.
 \begin{corol}\label{c:dp}
 	Let $f$ be compatible with a noise $Z$, and let $T$ be any measurable transformation from $\mbB \to \mcS$, where $\mcS$ is a measurable space. Then $T(f_D + \sigma Z)$ achieves $(\epsilon,\delta)$-DP over $\mcS$, where $\sigma$ is defined in Theorem \ref{t:finite}.
 \end{corol}
 Together, Theorem \ref{t:func} and Corollary \ref{c:dp} imply that the Gaussian mechanism gives very broad privacy protection for functional data and other infinite dimensional objects, as nearly any transformation or manipulation of the privacy enhanced release is guaranteed to maintain DP; this is known as a {\em postporocessing} property (e.g., see \citet{Dwork2014:AFD}).
 %
 %All the proofs can be found in Appendix~\ref{s:proof}.

 %%%%%%%%%%%%
 \section{Privacy for Mean Function Estimation} \label{s:g.s.bounds}
 %\subsection{Mean estimation} \label{s:m.gs.mean}
 In this section we consider the problem of estimating a mean function $\mu$ from a sample $X_1,\dots,X_n$ that are iid elements of $\mbH$ with $\E X_i = \mu \in \mbH$ and $\|X_i\|_\mbH \leq \tau < \infty$ for all $i$. We derive a bound on the global sensitivity as well as utility guarantees. In Section \ref{s:es} and in the Supplemental %and \ref{s:app} we 
 we will illustrate how to produce private releases of mean function estimates based on RKHS smoothing in more specific settings. In \citet{hall:rinaldo:wasserman:2013} one can also find examples for kernel density estimation and support vector machines.  
 
 As is usual in the DP literature, we assume that the data is standardized so that it is bounded, usually with $\tau =1$.  In this case, the sample mean $\hat \mu = n^{-1} \sum_{i=1}^n X_i$ is root-$n$ consistent and asymptotically normal \citep{kokoszka2017introduction}.
%%%% While we have thus far been fairly general, %in terms of $\mbB$, 
%%%% in Section \ref{s:es} and in the Supplemental %and \ref{s:app} we 
 %%%%we will illustrate how to produce private releases of mean function estimates based on RKHS smoothing in more specific settings. %{\color{red} MAYBE POINT OUT HERE AGAIN HOW ARE WE different from Hall here since we also focus on RKHS.}
%%% Here we derive a corresponding bound on the global sensitivity as well as utility guarantees. 
There are a multitude of methods for estimating smooth functions, however, a penalized approach is especially amenable to our privacy mechanism.
 In this case we define a penalty using the covariance of the noise, $C$.  However, the penalty and noise kernels need not be exactly the same, and in particular, we assume that penalty uses $ C^{\eta}$ for some $\eta \geq 1$.  Here $C^\eta$ has the same eigenfunctions as $C$, but the eigenvalues have been raised the power $\eta$.    This allows for greater flexibility in terms of smoothing and it is helpful for deriving utility guarantees. We define the penalized estimate of the mean $\mu$
 \[
 \hat \mu = \argmin_{m \in \mcH} %g(M)
 %\qquad \text{where}  \qquad g(m) = 
 \frac{1}{N} \sum_{i=1}^N \| X_i - m\|_{\mbH}^2 +\phi \| m \|_{\eta}^2,
 \]
 where $\phi$ is the \textit{penalty parameter}. The norm $\| \cdot \|_\eta$ is defined as the Cameron-Martin norm of $C^\eta$.  While the most natural candidate is $\eta = 1$, taking something slightly larger can actually help with statistical inference as we will see later on.
 Here, we can see the advantage of a penalized approach as it forces the estimate to lie in the space $\mcH$ which means that the compatibility condition, as discussed in theorems \ref{t:comp} and \ref{t:finite}, is satisfied.  A kernel different from the noise could be used, but one must be careful to make sure that the compatibility condition is met.  If $(\lambda_j, v_j)$ are the eigenvalue/function pairs of the $C$ and $\{X_{i}=\sum_{j=1}^{\infty} x_{ij}v_{j}:i=1,\dots,N\}$, with $x_{ij} = \langle X_i, v_j \rangle_{\mbH}$, then the estimate can be expressed as
 \begin{align}
 	\hat{\mu} = \frac{1}{N} \sum_{i=1}^{N} \sum_{j=1}^{\infty} \frac{\lambda_{j}^\eta}{\lambda_{j}^\eta+\phi} x_{ij}v_{j}, \label{e:deriv}
 \end{align}
 %see Appendix \ref{p:RKHS.s} for details. 
  We then have the following result.
 
 \begin{theorem} \label{t:RKHS.bound}
 	If the $\mbH$ norm of any element of the population is bounded by a constant $0 < \tau < \infty$ then the GS of $\hat \mu$ for $\eta \geq 1$ is bounded by
 	%\[\Delta_n^2 \leq \frac{4\tau^2}{N^2} \sum_{j=1}^\infty \frac{\lambda_j}{( \lambda_j  + \phi)^2}.\]
 	%{\bf NEW}
 	\[
 	\Delta_n^2 \leq \frac{4\tau^2}{N^2} \sup_j \frac{\lambda_j^{2\eta - 1}}{( \lambda_j^\eta   + \phi)^2}
 	\]
 	or more simply
 	\[
 	\Delta_n^2 \leq \frac{\tau^2}{N^2 \phi^{1/\eta}}
 	\left[ \frac{(2 \eta - 1)^{2-1/\eta}}{\eta^2} \right]
 	\leq \frac{4 \tau^2 }{N^2 \phi^{1/\eta}}.
 	\]
 \end{theorem}
 The resulting bound is practically very useful.  Data can be rescaled so that their $\mbH$ bound is, for example, 1, and then the remaining quantities are all tied to the used noise/RKHS.  Thus, the bound can be practically computed and the corresponding releases are guaranteed to achieve DP.% can be constructed.
 
 We conclude with a final theorem that provides a guarantee on the utility of $ \hat \mu + \sigma Z$.  One interesting note is that in finite dimensional problems, the magnitude of the noise added for privacy is often of a lower order than the statistical error of the estimate.  However, in infinite dimensions, this is no longer true unless $\eta >1$.  This is driven by the fact that the squared bias is of the order $\phi$, and thus $\phi$ must go to zero like $N^{-1}$ if it is to balance the variance of $\hat \mu$.  However, in that case the magnitude of the noise added for privacy is of the order $\sigma^2 \asymp N^{-2+1/\eta}$.  If $\eta=1$, then $\sigma^2$ is also of the order $N^{-1}$, while if $\eta > 1$, then it is of a lower order and thus asymptotically negligible.   We remind the reader that the noise and thus $C$ is arbitrary, so $\eta$ can be chosen in a way that is appropriate for $\mu$ by using a rougher noise.
 
 \begin{theorem}
 	Assume the $X_i$ are iid elements of $\mbH$ with norm bounded by $\tau < \infty$.  Define
 	\[
 	\tilde \mu := \hat \mu + \sigma Z,
 	\]
 	where
 	\[
 	\sigma^2 
 	= \left[\frac{2 \log(2/\delta)}{ \epsilon^2 } \right] \times \left[\frac{\tau^2 (2\eta - 1)^{2-1/\eta}}{  N^2 \phi^{1/\eta} \eta^2} \right].
 	\]
 	If the tuning parameter, $\phi$, satisfies $ \phi \propto N^{-1}$ and if $\|\mu\|_\eta < \infty$ then we have
 	\[
 	\E\| \tilde \mu - \hat \mu\|_\mbH^2 = o(N^{-1}) \qquad \text{and} \qquad \E\| \tilde \mu -  \mu \|^2_{\mbH} = O\left(N^{-1} \right),
 	\]
 	while $\tilde \mu$ achieves ($\epsilon$-$\delta$) DP in $\mbH$.
 \end{theorem}

 \section{Empirical Study} \label{s:es}
Here we briefly present simulations with $\mbB = L^2[0,1]$ to explore the impact of parameters on the utility of sanitized releases. We consider the problem of estimating the mean function from a random sample of functional observations using RKHS smoothing, as discussed in Section \ref{s:g.s.bounds}.
%
%The quantities we will explore include the kernel function $C(t,s)$, used to define $\mcH$, the privacy parameters, $(\epsilon,\delta)$, and the \textit{penalty parameter}, $\phi$, that (combined with $K$) controls the level of smoothing. For simplicity, we assume that $\eta=1$, but we will explore this choice further in Section \ref{s:face}.

For the RKHS, $\mcH$, we consider the Gaussian (squared exponential) kernel : 
\begin{align} \label{Kernels}
C_{1}(t,s) & =\exp\left\{\dfrac{-|t-s|^2}{\rho} \right\} 
\end{align}
We simulate data using the Karhunen-Loeve expansion, a common approach in FDA simulation studies.  In particular we take
\begin{align}\label{KL}
X_{i}(t)=\mu(t)+\sum_{j=1}^{m} j^{-p/2} U_{ij}v_{j}(t) \ \ \ \ \ \ t \in [0,1],
\end{align}
where the scores, $U_{ij}$, are drawn iid uniformly between $(-0.4,0.4)$. %\textcolor{red}{[In the code, I considered m as the number of of positive kernel eigenvalues.]}  
The functions, $v_j(t)$, are taken as the eigenfunctions of $C_1$ %(though each of the kernels lead to very similar eigenfunctions) 
and $m$ was taken as the largest value such that $\lambda_m$ was numerically different than zero in {\tt R} (usually about $m=50$). All of the curves are generated on an equally spaced grid between $0$ and $1$, with 100 points and the RKHS kernel and the noise kernel will be taken to be the same (i.e. $\eta = 1$). The range parameter for the kernel used to define $\mcH$ is taken $\rho = 0.001$ and the smoothness parameter of the $X_i(t)$ is set to $p=4$ . The mean function, sample size and DP parameters will also be set as $\mu(t)=0.1 \sin(\pi t)$, $N=25$, $(\epsilon=1 , \delta=0.1)$, respectively. We vary the penalty, $\phi$, from $10^{-6}$ to $1$ to consider its effect.

\textcolor{black}{Note that we take $\tau=\sup\|X_i\|_\mbH$ for any $i \in 1 , \hdots , N$ and thus all qualities needed for Theorem \ref{t:RKHS.bound} are known.} The risk is fixed by choosing the $\epsilon$ and $\delta$ in the definition of DP.  We thus focus on the utility of the privacy enhanced curves by comparing them graphically to the original estimates.
% and the functional data. {\color{red}I AM NOT SURE WHAT is meant here by comparing to the data? We are only looking a the mean estimate of the original data and then the mean estimate of the private data, we never plot private sample curves? So what I am missing here? }
Ideally, the original estimate will be close to the truth and the privacy enhanced version will be close to the original estimate.  What we will see is that by compromising slightly on the former, one can makes substantial gains in the latter. %so allowing for a bit more error in the original estimate, we will do much better with the private, so that's not surprising, so what's really the big point /msg here???!?

In Figure \ref{DP_pen1} we plot all of the generated curves in grey, the RKHS smoothed mean in green, and the sanitized estimate in red. We can see that as the penalty increases, both estimates shrink towards each other and to zero.  There is a clear ``sweet spot" in terms of utility, where the smoothing has helped reduce the amount of noise one has to add to the estimate while not over smoothing. Further simulations that explore the impact of different parameters can be found in the supplemental \ref{s:ees}.  % \ref{s:ees}.

\begin{figure}[ht!]
	\centering
	\includegraphics[width=1\columnwidth]{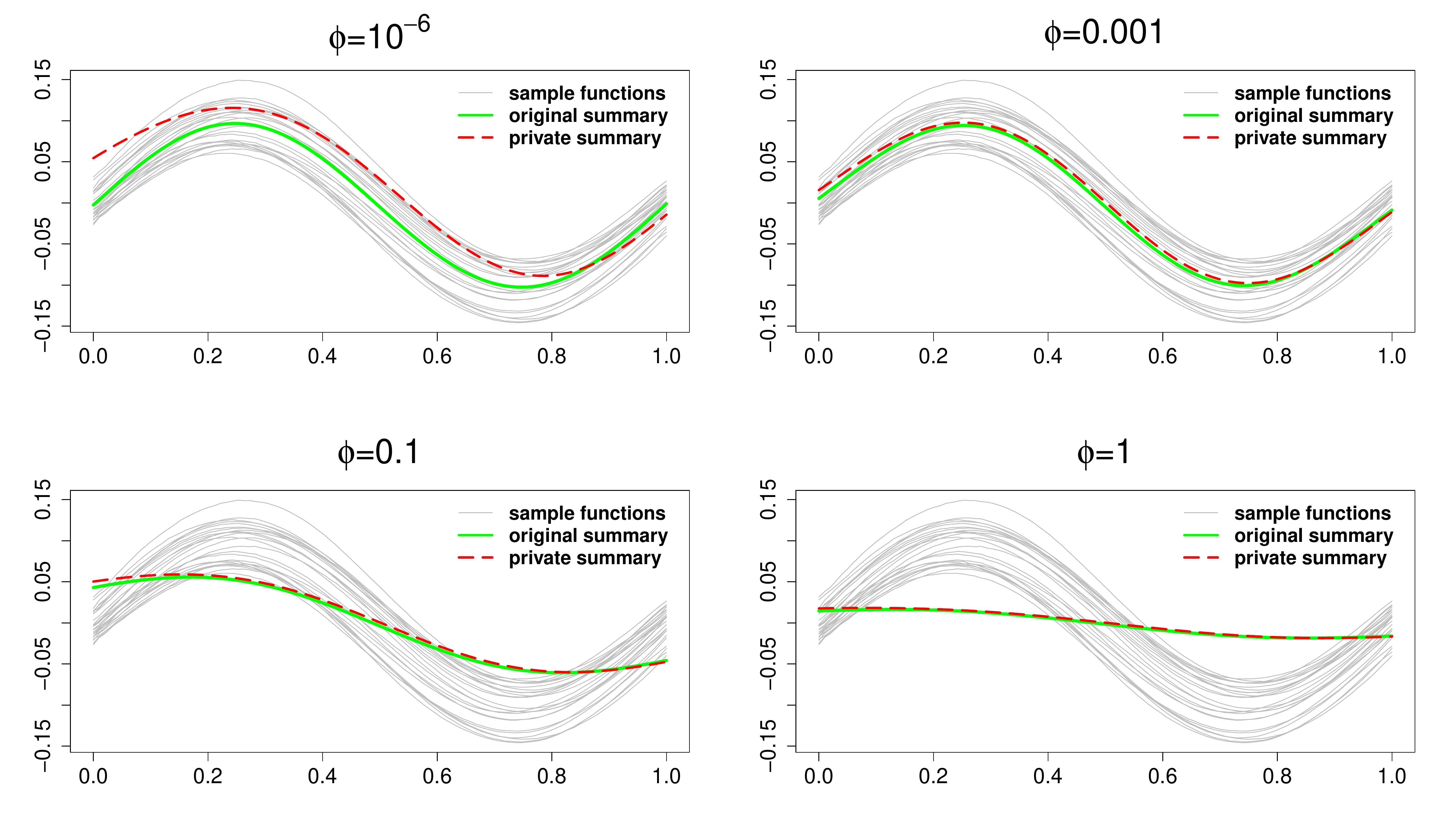}
	\caption{Original and private RKHS smoothing mean with Gaussian Kernel ($C_{1}$) for different values of penalty parameter $\phi$ }
	\label{DP_pen1}
\end{figure}

 \section{Applications}  \label{s:app}
%We apply our methods to two popular datasets in the FDA literature.  The first concerns Diffusion Tensor Imaging and is part of the {\tt Refund} package in {\tt R}, while the second is the Berkeley growth data, which can be found in the {\tt FDA} package in {\tt R}.
%\subsection{Diffusion Tensor Imaging}
%\textcolor{red}{IS THERE an interesting privacy feature with either of the two examples that we could discuss? e.g., why would this be interesting example from privacy perspective?}
%Regression with Functional Data or {\tt refund} \citep{refund:2016} is an {\tt R} package that includes a variety of tools related to both linear and nonlinear regression for functional data. 

In this section we illustrate our method on an application involving brain scans ({\em diffusion tensor imaging}, DTI) that gives fractional anisotropy (FA) tract profiles for the corpus callosum (CCA) and the right corticospinal tract (RCST) for patients with multiple sclerosis as well as controls; data are part of the {\tt refund} \citep{refund:2016} {\tt R} package. 
%of patients with multiple sclerosis as part of the {\tt refund} \citep{refund:2016} {\tt R} package that gives fractional anisotropy (FA) tract profiles for the corpus callosum (CCA) and the right corticospinal tract (RCST) for patients with multiple sclerosis as well as controls. 
This type of imaging data is becoming more common and the privacy concerns can be substantial.  Images of the brain or other major organs might be quite sensitive source of information, especially if the study is related to some complex disease such as cancer, HIV, etc.  Thus it is useful to illustrate how to produce privacy enhanced versions of function valued statistics such as mean functions.
We focus on the \textit{CCC} data, which includes 382 patients measured at 93 equally spaced locations along the CCA.%, and use them to illustrate our proposed methodology.

Our aim is to release a sanitized RKHS estimate of the mean function. We consider three kernels $C_{1}$, $C_{3}$ and $C_{4}$ which correspond to the Gaussian kernel, Mat\'ern kernel with $\nu=3/2$, and the exponential kernel, respectively. Each kernel is from the Mat\'ern family of covariances \citep{stein2012interpolation}.  The exact forms are given in \eqref{Kernels} in the supplement, where a fourth kernel $C_2$ is also considered that is ``inbetween" $C_1$ and $C_3$ (hence the odd numbering). In all settings we take $(\epsilon,\delta)=(1,0.1)$ and select the penalty, $\phi$, and range parameter, $\rho$, according to two different approaches.  The first is regular \textit{Cross Validation}, CV, and the second we call \textit{Private Cross Validation}, PCV. In CV we fix $\phi$ and then take the $\rho$ that gives the minimum 10-fold cross validation score. We do not select $\phi$ based on cross validation because, based on our observations, the minimum score is always obtained at the  minimum $\phi$ for this data. In PCV we take nearly the same approach, however,
when computing the CV score we take the \textit{expected} difference (via Monte-Carlo) between our privacy enhanced estimate and the left out fold from the original data.  In other words, we draw a sample of privacy enhanced estimates, compute a CV score for each one, and then average the CV scores.     %{\color{red} SO HERE the left out fold is the ORIGINAL DATA?}
In our simulations we use 1000 draws from the distribution of the sanitized estimate.  We then find both the $\phi$ and $\rho$ which give the optimal PCV score based on a grid search.

%We begin by presenting the results for CV.  
For the CV-based results, for each of the kernels, we fixed a value for  $\phi \in \{ 0.0001, 0.001, 0.01, 0.03 \}$ and then vary the $\rho$ between $[0.01,2]$.
We use the optimal parameter values in Table \ref{Par_CV_Reg1} to produce the privacy enhanced estimates for $C_{1}$ in Figure \ref{DTI_Gau_R1}. We see that the utility of the privacy enhanced versions increases as $\phi$ increases, however, the largest values of $\phi$ produce estimates that are over smoothed. There is a good trade-off between privacy and utility with $\phi=0.01$ for $C_{1}$. The results for other kernels are reviwed in supplemental \ref{s:eapp}.

\begin{figure}[ht!]
	\centering
	\includegraphics[width=1\columnwidth]{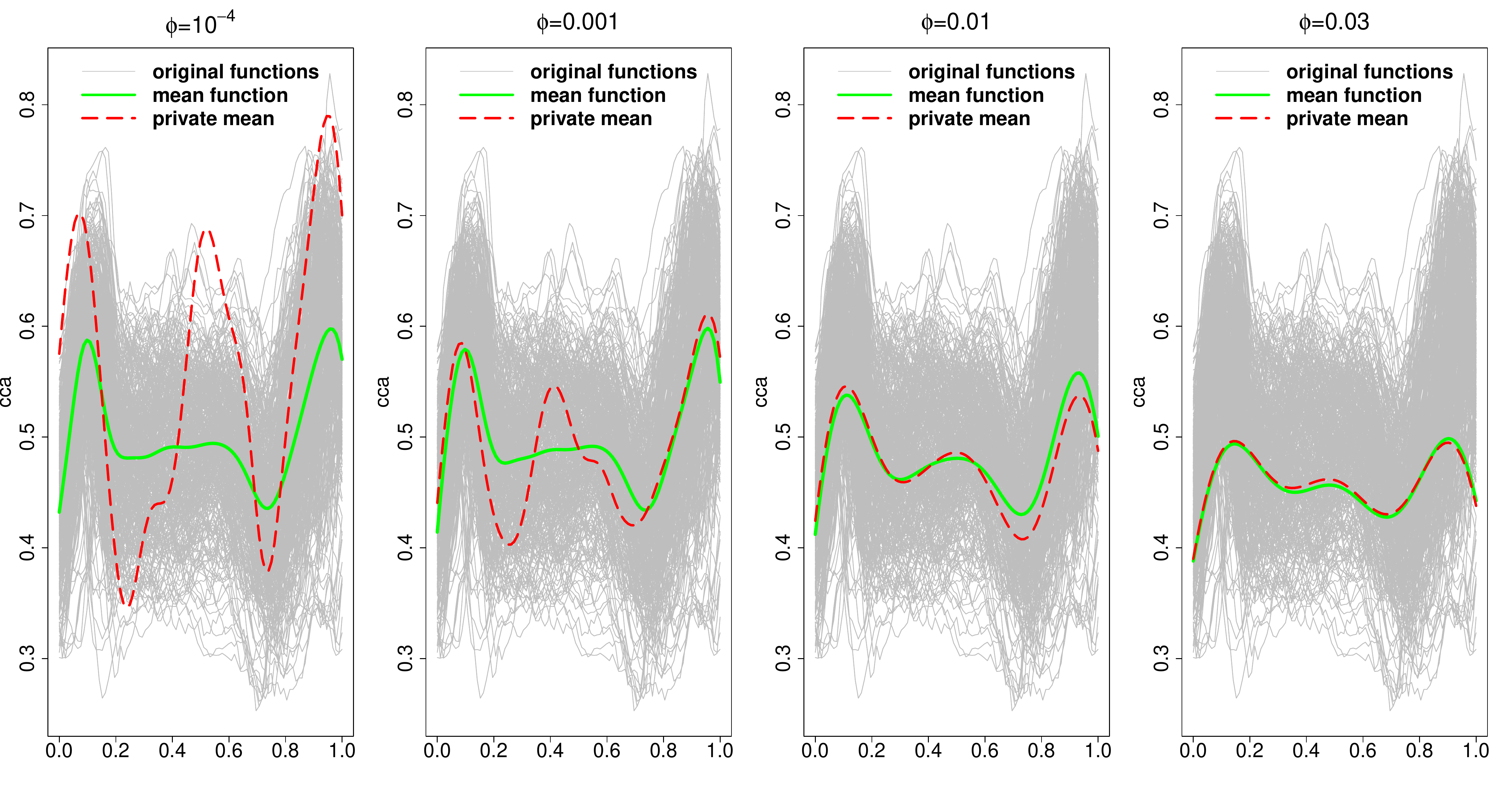}
	\caption{Mean estimate for cca and its private release using Gaussian kernel ($C_{1}$) with CV.}
	\label{DTI_Gau_R1}
\end{figure}

Turning to PCV, we varied $\phi$ in range $[10^{-4} , 0.1]$ for each of the kernels but $\rho$ will be varied in $[0.01 , 0.1]$, $[0.05 , 0.5]$ and $[0.2,1]$ for $C_1$,$C_3$ and $C_4$ respectively. Here we use the optimal parameters in Table \ref{Par_CV_Irr1} to generate privacy enhanced estimates, given in Figure \ref{DTI_Irr1}.
Here we see that the utility of the privacy enhanced estimates is excellent for $C_1$. Using PCV tends to over smooth the original estimates (green lines), however, by slightly over smoothing we make substantial gains in utility as we add less noise.

\begin{figure}[ht!]
	\centering
	\includegraphics[width=1\columnwidth]{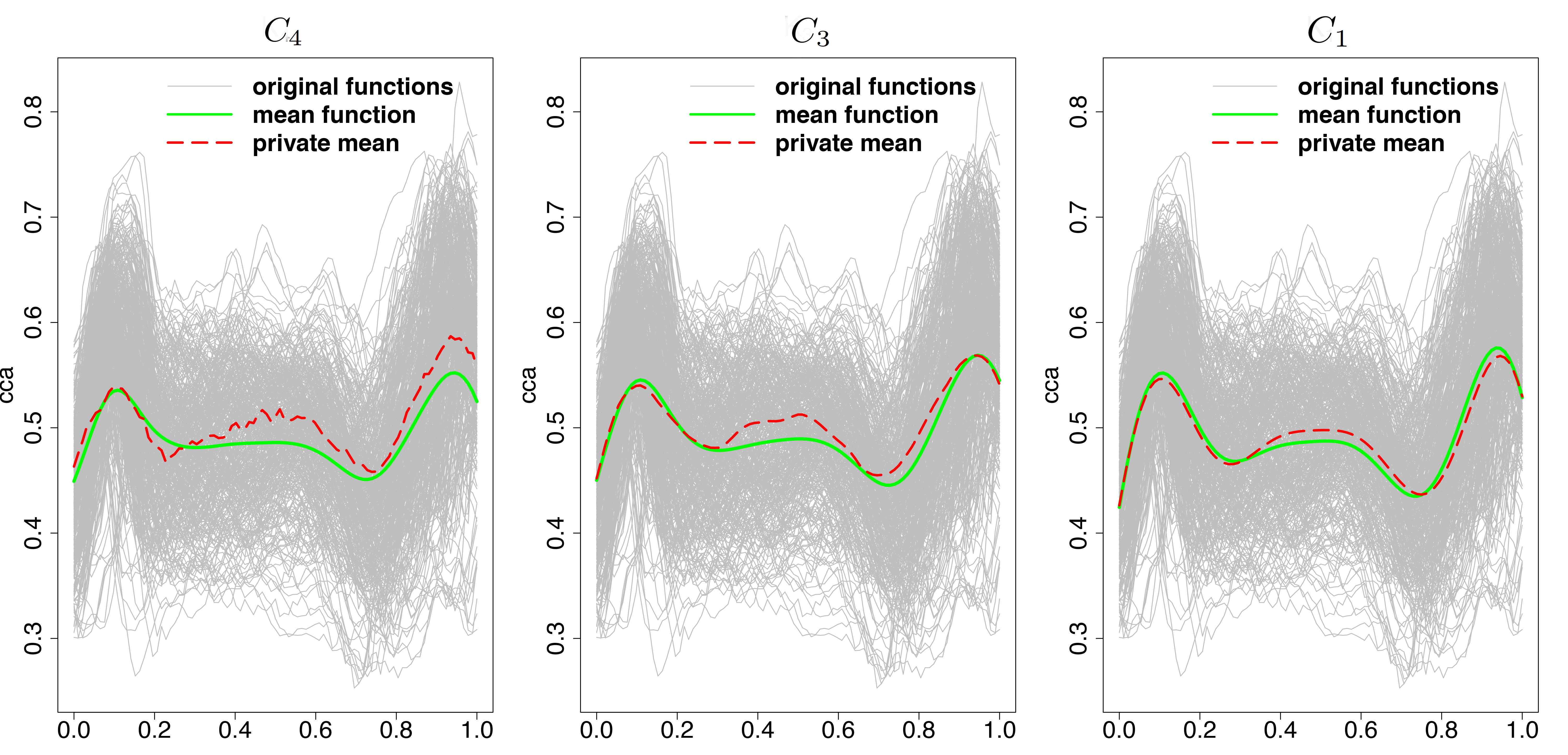}
	\caption{Mean estimate of CCA and its private release for Exponential ($C_{4}$), Mat\'ern$(3/2)$ ($C_{3}$) and Gaussian kernels ($C_{1}$) using PCV.}
	\label{DTI_Irr1}
\end{figure}
 
 \section{Conclusions}
 In this work we have provided a mechanism for achieving $(\epsilon,\delta)$-DP for a wide range of summaries related to functional parameter estimates.  This work expands dramatically upon the work of \citet{hall:rinaldo:wasserman:2013,alda2017bernstein,smith2018differentially}, who explored this topic in the context of point-wise releases of functions.  Our work covers theirs as a special case, but also includes path level summaries, full function releases, and nonlinear releases quite broadly.%, that have not been considered before in statistical privacy literature. %{\color{red} DO WE NEED TO REF AGAIN HERE THE OTHER TWO NEW REFS AND HOW ARE CONTRIBUTIONS DIFFERENT?}
 In general, functional data can be highly identifiable compared to scalar data. In biomedical settings, for example, a study may collect and analyze  many pieces of information such as genomic sequences, biomarkers, and biomedical images,  which either alone or linked with each other and demographic information, lead to greater disclosure risk% ; for a recent example of potential identification risks related to such data, see  
 \citep{Lippert19092017}.
 
 The heart of our work utilizes densities for functional data in a way that has not yet been explored in the functional data literature.  Previously, usable densities for functional data were thought not to exist \citep{delaigle:hall:2010} and researchers instead relied on various approximations to densities.  We showed how useful forms for densities can be constructed and utilized.  However, it is still unclear how extensively these densities can be used for other FDA problems.  %At the heart of this issue is the \textit{orthogonality of probability measures}; in infinite dimensional spaces dealing with the orthogonality of measures becomes a major issue, which has a direct impact on the utility of density based methods for FDA.
 
 The literature on privacy for scalar and multivariate data is quite extensive, while there has been very little work done for FDA and related objects.  Therefore, there are many opportunities for developing additional theory and methods for such complicated data.  One issue that we believe will be especially important is the role of smoothing and regularization in preserving the utility of privacy enhanced releases.  As we have seen, a bit of extra smoothing can go a long way in terms of maintaining privacy, however, the type of smoothing may need to be properly tailored to the application for much complicated objects. %, such as faces.

% \newpage  

\bibliography{privacy}
\bibliographystyle{icml2019}

\appendix
\clearpage
\setcounter{page}{1}
\begin{center}
	\bf{\large Supplemental Material}
\end{center}

%%%%%%%%%%%%
\section{Proofs}

\begin{proof}[Proof of Theorem \ref{t:finite}]
	For $f_i \in \mcK$ define
	\[
	\nu_D^\top = (f_1(\theta_D). \dots, f_N(\theta_D)),
	\]
	and the matrix $\bK = \{ \langle f_i, f_j \rangle_{\mcK} \} $; recall $C(f_i,f_j) = \langle f_i, f_j \rangle_{\mcK}$.  Using Proposition 3 of \citet{hall:rinaldo:wasserman:2013}, we then need only to show that
	\[
	(\nu_D - \nu_{D'})^\top \bK^{+} (\nu_D - \nu_{D'}) \leq \| \theta_D - \theta_{D'}\|^2_{\mcH},
	\]  
	where $+$ denotes the Moore-Penrose generalized inverse.  We take a common strategy to such problems by showing that the left hand side can be expressed as $\| P(\theta_D - \theta_{D'})\|^2_{\mcH}$, where $P$ is a projection operator.  Recall that we can move between $\mcK$ and $\mcH$ via the transformation $ h =  C T_h$ for $h \in \mcH$ and $T_h \in \mcK$.      
	Define the operator, $P_1: \mbB \to \Span\{f_1,\dots,f_N\} \subset \mcK$ as
	\[
	P_{1}(x)=\sum_{i=1}^{N}f_{i}\sum_{j=1}^{N} (K^{
	+})_{ij} f_{j} (x)
	\]
	and its analog into $\mcH$, $P: \mbB \to \Span\{C(f_1),\dots,C(f_N)\}$:
	\[
	P(x) = C(P_1(x)) = \sum_{i=1}^{N}C(f_{i})\sum_{j=1}^{N} (K^{+})_{ij} f_{j} (x).
	\]
	Notice that while $P_1$ maps elements of $\mbB$ to $\mcK \supset B^*$, $P_2$ maps elements of $\mbB$ into the Cameron-Martin space, $\mcH \subset \mbB$.
	By the reproducing property, there exists $T_{\theta_D - \theta_{D'}} \in \mcK$ such that %$C T_{\theta_D - \theta_{D'}} = \theta_D -\theta_{D'}$ and 
		\begin{align*}
		&\langle T_{\theta_D - \theta_{D'}}, P_1(\theta_D - \theta_{D'}) \rangle_{\mcK} \\
		& = \sum_{i=1}^{N} \langle f_{i}, T_{\theta_D - \theta_{D'}} \rangle_{\mcK} \sum_{j=1}^{N} (K^{+})_{ij} f_{j} (\theta_D - \theta_{D'} )\\
		& = \sum_{i=1}^{N} f_{i}(\theta_D - \theta_{D'})  \sum_{j=1}^{N} (K^{+})_{ij} f_{j} (\theta_D - \theta_{D'} )\\
		& = (\nu_D - \nu_{D'})^\top \bK^{+} (\nu_D - \nu_{D'}).
		\end{align*}
	Moving to $\mcH$ we have
	\[
	 \langle T_{\theta_D - \theta_{D'}}, P_1(\theta_D - \theta_{D'}) \rangle_{\mcK}
	 = \langle \theta_D - \theta_{D'}, P(\theta_D - \theta_{D'}) \rangle_{\mcH}.
	\]
	If we show that P is a projection operator over $\mcH$, i.e., symmetric and idempotent, we will have the desired bound.
	
	First, $P$ is \textbf{idempotent} by direct verification:
	\begin{align*}
	P^2(x) & = \sum_{i=1}^{N}C(f_{i})\sum_{j=1}^{N} (K^{+})_{ij} f_{j} \left( \sum_{k=1}^{N}C(f_{k})\sum_{l=1}^{N} (K^{+})_{kl} f_{l}(x) \right)
	\\
	& = \sum_{i=1}^{N}C(f_{i})\sum_{j=1}^{N} (K^{+})_{ij} \sum_{k=1}^{N}  C(f_{k}, f_{j}) \sum_{l=1}^{N} (K^{+})_{kl} f_{l}(x)
	\\
	& = \sum_{i=1}^{N}C(f_{i}) \sum_{l=1}^{N} (K^{+})_{il} f_{l}(x)
	= P(x).
	\end{align*}
	Second, we show $P$ is \textbf{symmetric} with respect to the $\mcH$ inner product by making repeated use of the reproducing property:
	\begin{align*}
	\langle P(x) , y \rangle_{\mcH} & = \langle P_1(x) , T_y \rangle_{\mcK}  \\
	& = \sum_{i=1}^{N}f_{i}(y)\sum_{j=1}^{N} (K^{-1})_{ij} f_{j} (x) \\
	& = \langle T_x , P_1(y) \rangle_{\mcK} = \langle x, P(y) \rangle_{\mcH}.
	\end{align*}
	Hence P is a projection operator from $\mbH$ to $\mbK$, and the claim of the theorem holds.
\end{proof}

%%%%%%%%%%%%%%%%%%%%%%%%%%
\begin{proof}[Proof of Theorem \ref{t:func}]
	We aim to show that for any measurable subset $A \subset \mbB$ we have
	\[
	P_D(A) \leq e^{\epsilon}P_{D^\prime}(A) + \delta,
	\]
	where $P_D$ denotes the measure of $\tilde f_D$.  
	Recall the global sensitivity for the functional case is
	\[
	\Delta^2=\sup_{D \sim D^\prime} \| f_D - f_{D^\prime}\|_\mcH^2 .
	\]
	The density of $\tilde f_D$ wrt $\sigma Z$ is
	\[
	\exp\left\{-\frac{1}{2\sigma^2}(\| f_D\|_{\mcH}^2  - 2 T_D(x) )\right\},
	\]
	where for simplicity we denote $T_D = T_{\theta_D}$.  
	We equivalently aim to show that
	\begin{align*}
		P_D(A) & = \int_{A} dP_{D}(x)  = \int_{A} \frac{dP_D}{dP_{D^\prime}} (x) dP_{D^\prime}(x)   \\ 
	& \leq e^{\epsilon} \int_{A} dP_{D^\prime}(x) + \delta.
	\end{align*}
	We can express
	\begin{align*}
	& \frac{d P_D}{ d P_{D^\prime}}(x) = \frac{d P_D}{ d Q}(x) / \frac{d P_{D^\prime}}{ d Q}(x)  \\
	& = \exp\left\{-\frac{1}{2\sigma^2}(\| f_D\|_{\mcH}^2   - \| f_{D^\prime} \|_{\mcH}^2
	-2 (T_{D} - T_{D^\prime})(x)   ) \right\}.
	\end{align*}
	Expand
	\begin{align*}
	& \|f_{D^\prime}\|_\mcH^2
	= \|f_{D^\prime} - f_{D} + f_{D} \|_\mcH^2 \\
	& = \|f_{D^\prime} - f_{D}\|_{\mcH}^2 + \| f_{D} \|_{\mcH}^2 - 2\langle f_{D} - f_{D^\prime},f_{D} \rangle_{\mcH},
	\end{align*}
	and recall that we can write $\langle x,y \rangle_{\mcH} = T_x(y)$.
	So we have
	\begin{align*}
	& \frac{d P_D}{ d P_{D^\prime}}(x) =  \\
	& \exp\left\{-\frac{1}{2\sigma^2}(- \|f_{D} - f_{D^\prime}\|_\mcH^2 - 2 (T_{D} - T_{D^\prime})(x -  f_{D})) \right\}.
	\end{align*}
	Decompose $\mbB = \mcH_1 \bigcup \mcH_2$ where for $x \in \mcH_1$ we have $\dfrac{d P_D}{ d P_{D^\prime}}(x) \leq e^{\epsilon}$ and for $x \in \mcH_2$ we have $\dfrac{d P_D}{ d P_{D^\prime}}(x) > e^{\epsilon}$.  Then trivially we have that
	\[
	P_D(A) = P_D(A \cap \mcH_1) +  P_D(A \cap \mcH_2).
	\]
	Using the definition of $\mcH_1$ we have that
	\begin{align*}
		& P_D(A \cap \mcH_1)
	 = \int_{A \cap \mcH_1} \frac{d P_D}{ d P_{D^\prime}}(x) \frac{d P_{D
			\prime }}{ d Q}(x) \ d Q(x) \\
			& \leq e^{\epsilon} \int_{A \cap \mcH_1} \frac{d P_{D^
			\prime }}{ d Q}(x) \ d Q(x) 
	 \leq e^\epsilon P_{D^\prime}(A).
	\end{align*}

	The proof will be complete if we can show that
	\[
	P_D(A \cap \mcH_2) \leq \delta.
	\]
	This is equivalent to showing that
	\begin{align*}
	& P\biggr( -\frac{1}{2\sigma^2}(- \|f_{D} - f_{D^\prime}\|_\mcH^2  \\ & - 2 (T_{D} - T_{D^\prime})(X -  f_{D}))  \geq \epsilon \biggl)
	\leq \delta,
	\end{align*}
	where $X \sim N(0, \sigma^2 C)$.  This can equivalently be stated as
	\begin{align*}
	&-\frac{1}{2\sigma^2}(- \|f_{D} - f_{D^\prime}\|_\mcH^2 - 2 (T_{D} - T_{D^\prime})(X -  f_{D})) \geq \epsilon \\
	& \Leftrightarrow    (T_{D} - T_{D^\prime})(X)  \geq  \sigma^2\left[\epsilon  - \frac{1}{2 \sigma^2}\|f_{D^\prime} - f_D\|_\mcH^2\right]   %\\
	%\Leftrightarrow &  \langle T_{f_{D} - f_{D^\prime}},X  \rangle_{\mcK}  \geq  \sigma^2\left[\epsilon  - \frac{1}{2\sigma^2}\|f_{D} - f_{D^\prime}\|_\mcH^2 \right].
	\end{align*}
	However $ (T_{D} - T_{D^\prime})(X)$ is a normal random variable with mean zero and variance $\| f_D - f_{D'}\|_{\mcH}^2$.  
	%\[
	%\sigma^2 \langle C (C^{-1}(f_{D^\prime} - f_D)), C^{-1}(f_{D^\prime} - f_D) \rangle_{\mcK}
	%= \delta^2 \| f_{D^\prime} - f_D\|_\mcH \leq \delta^2 \Delta^2.
	%C(T_{f_{D} - f_{D^\prime}} ,T_{f_{D} - f_{D^\prime}} ) = \| f_D - f_{D'}\|^2_{\mcH}.
	%\]
	So, if $Z \sim N(0,1)$ then we have that
	\begin{align*}
	& P\left( -\frac{1}{2\sigma^2}(- \|f_{D} - f_{D^\prime}\|_\mcH^2 - 2 (T_{D} - T_{D^\prime})(X -  f_{D}))  \geq \epsilon \right) \\
	& \leq P\left( \sigma \Delta Z \geq \sigma^2\left[\epsilon  - \frac{1}{2\sigma^2}\|f_{D} - f_{D^\prime}\|_\mcH^2 \right] \right) \\
	& \leq P\left( Z \geq \frac{\sigma}{\Delta}\left[\epsilon  - \frac{\Delta^2}{2\sigma^2} \right] \right)\\
	& = P\left(  Z \geq    \sqrt{2 \log(2/\delta)}  - \frac{\epsilon }{2\sqrt{2 \log(2/\delta)}}\right)  \\
	& \leq P\left(  Z \geq    \sqrt{2 \log(2/\delta)}  - \frac{1}{2\sqrt{2 \log(2/\delta)}}\right)  \leq \delta
	\end{align*}
	as long as $\epsilon \leq 1$ \citep{hall:rinaldo:wasserman:2013}.
	
\end{proof}

\subsection{Derivation of RKHS Estimate} \label{p:RKHS.s}
Recall that
\[
g(m) = \frac{1}{N} \sum_{n=1}^N \| X_n - m\|_{\mbH}^2 + \phi \| m\|_{\eta}^2.
\]
Without loss of generality, we may drop any terms not involving $m$ and write
\begin{align*}
g(m) & = - 2 \langle \bar X, m \rangle_{\mbH} + \|m\|_{\mbH}^2 + \phi \|m\|_{\eta}^2 \\
& = - 2 \langle \bar X, m \rangle_{\mbH} + \langle m, m \rangle_{\mbH} + \phi \langle m, C^{-\eta} m \rangle_{\mbH} .
\end{align*}
Since we are working with a Hilbert space, it can be identified with its own dual.  We transfer everything over to the Cameron-Martin Space of $C^\eta$, call it $\mcH_\eta$, which contains $\mcH$:
\begin{align*}
 g(m)  & =  - 2 \langle \bar X, C^\eta  m \rangle_{\mcH} \\
& + \langle m , C^\eta m \rangle_{\mcH} + \phi \langle   m, m \rangle_{\mcH}.
\end{align*}
We then have that
\[
g'(m) = -2 C^\eta \bar X + 2 C^\eta m + 2 \phi  m.
\]
Setting the above equal to zero we have that
\begin{align} \label{p:mu.hat}
C^\eta \bar{X}=C^\eta \hat{\mu}+\phi \hat \mu .
\end{align}
or
\[
\hat \mu = ( C^\eta + \phi I )^{-1} C^\eta (\bar X).
\]
Since $(\lambda_{j},v_j)$ are the eigenvalues/eigenfunctions of $C$ and $X_{i}=\sum_{j=1}^{\infty} x_{ij}v_{j}$ then we have
\begin{align*}
\hat{\mu}&  =\sum_{j=1}^\infty \langle \hat{\mu}, v_j \rangle_{\mbH} v_j \nonumber 
 = \sum_{j=1}^\infty \frac{ \lambda_j^\eta }{\lambda_j^\eta + \phi } \langle \bar X, v_j \rangle_{\mbH} v_j \nonumber \\
& = \frac{1}{N} \sum_{i=1}^{N} \sum_{j=1}^{\infty} \frac{\lambda_{j}^\eta}{\lambda_{j}^\eta+\phi} x_{ij}v_{j}. \nonumber
\end{align*}

\begin{proof}[Proof of Theorem \ref{t:RKHS.bound}]
	The upper bound for $\Delta_n^2 $ is derived as following:
	\begin{align*}
	\Delta_n^2 & =\sup_{D \sim D^{\prime} } \| \hat{\mu}_{D} - \hat{\mu}_{D^{\prime}} \|_{\mcH}^{2}  \\
	& = \sup_{D \sim D^{\prime} } \| \frac{1}{N} \sum_{j=1}^{\infty} \frac{\lambda_{j}^\eta}{(\lambda_{j}^\eta+\phi)}(x_{1j}-x_{1^{\prime}j})v_{j} \|_{\mcH}^{2} \\
	%& = \sup_{D \sim D^{\prime} } \sum_{j=1}^{\infty} \dfrac{ \langle \frac{1}{N} \frac{\lambda_{j}}{(\lambda_{j}+\phi)} (x_{1j} - x_{1j}^{\prime})v_{j} , v_{j} \rangle_{\mcK}^{2}}{\lambda_{j}}  \\
	& \leq \frac{1}{N^2} \sup_j \frac{\lambda_j^{2\eta -1 }}{(\lambda_j^\eta + \phi)^2}\sup_{D \sim D^{\prime}} \sum_{j=1}^{\infty}  \langle X_{1} - X_{1}^\prime , v_j\rangle_{\mbH}^{2}  \\
	& =\frac{1}{N^2} \sup_j \frac{\lambda_j^{2\eta-1}}{(\lambda_j^\eta + \phi)^2} \sup_{D \sim D^{\prime}} \| X_1 - X_1^\prime \|_{\mbH}^{2}  \\
	& \leq \frac{4 \tau^2}{N^2} \sup_j \frac{\lambda_j^{2\eta -1 }}{(\lambda_j^\eta + \phi)^2}.
	\end{align*}
	We can also derive a simpler bound by examining the function
	\[
	f(x) = \frac{x^{2 \eta -1}}{(x^\eta+\phi)^2}, \qquad x \geq 0,
	\]
	and where it attains its maximum.  Taking the derivative we have $f'(x) = 0$ if and only if 
	\begin{align*}
	& (x^\eta + \phi)^2 (2 \eta - 1) x^{2 \eta -2} - x^{2\eta -1} 2 \eta x^{\eta-1} (x^\eta + \phi) = 0 \\
	& (x^\eta + \phi) (2 \eta - 1)  -  2 \eta x^{\eta} =0 \\
	& x = (\phi ( 2 \eta -1 ))^{1/\eta}.
	\end{align*}
	Taking a second derivative shows that this is where the maximum occurs.  We then have that
	\[
	f(x) \leq \frac{(\phi(2 \eta -1))^{2 - 1/\eta}}{(\phi(2 \eta -1) + \phi)^2}
	= \phi^{-1/\eta} \frac{(2 \eta - 1)^{2 - 1/ \eta}}{4 \eta^2}
	\]
	Thus, we can also use the bound
	\[
	\frac{4 \tau^2}{N^2} \sup_j \frac{\lambda_j^{2\eta-1}}{(\lambda_j^{\eta} + \phi)^2}
	\leq \frac{\tau^2}{N^2 \phi^{1/\eta}} \frac{(2 \eta - 1)^{2 - 1/ \eta}}{\eta^2}.
	\]
	For $\eta=1$, the bound becomes $\tau^2 N^{-2} \phi^{-1}$, while another calculus argument shows that regardless of $\eta$, one will always have
	\[
	\frac{\tau^2}{N^2 \phi^{1/\eta}} \frac{(2 \eta - 1)^{2 - 1/ \eta}}{\eta^2}
	\leq \frac{4\tau^2}{N^2 \phi^{1/\eta}},
	\] 
	as desired.
	
\end{proof}

%%%%%%%%%%%%
\section{Extension of Empirical Study} \label{s:ees}
In this section we review the impact of different parameters on the utility of sanitized releases introduced in Section \ref{s:es}.
For the RKHS, $\mcH$, we would consider four popular kernels:
{\small
\begin{align} \label{Kernels}
C_{1}(t,s)&=\exp\left\{\dfrac{-|t-s|^2}{\rho} \right\}  \\
C_{2}(t,s)&=\left(1+\dfrac{\sqrt{5}|t-s|}{\rho}+\dfrac{5(t-s)^2}{3\rho^2} \right)\exp\left\{\dfrac{-\sqrt{5}|t-s|}{\rho}\right\} \nonumber \\
C_{3}(t,s)&=\left(1+\dfrac{\sqrt{3}|t-s|}{\rho} \right)\exp\left\{\dfrac{-\sqrt{3}|t-s|}{\rho}\right\} \nonumber \\
C_{4}(t,s)&=\exp\left\{\dfrac{-|t-s|}{\rho}\right\}. \nonumber
\end{align}
}%
the first is also known as the Gaussian or squared exponential kernel and the last is also known as the exponential, Laplacian, or Ornstein-Uhlenbeck kernel.

Recall the all parameters discussed in Section \ref{s:es} will be fixed in all scenarios, except for the one where they are explicitly varied to consider their effect.

The scenario 1 was discussed in Section \ref{s:es}.
\subsubsection*{Scenario 2: Varying kernel range parameter $\rho$}
Here all defaults are used except the range parameter for the noise and RKHS (which are taken to be the same in all settings) that ranges from $0.002$ to $2$.  The results are presented in Figure \ref{DP_ro}.  We see very similar patterns to Scenario 1, where increasing $\rho$ increases the smoothing of both the estimate and its privacy enhanced version.  However, increasing $\rho$ smooths more than it shrinks and there is still a non-negligible difference between the two estimates for larger values, (e.g., $\rho=0.2$).  Practically, both $\rho$ and $\phi$ should be chosen together for the best performance, which we will explore further in Section \ref{s:app}.
\begin{figure}[ht!]
	\centering
	\includegraphics[width=0.85\columnwidth]{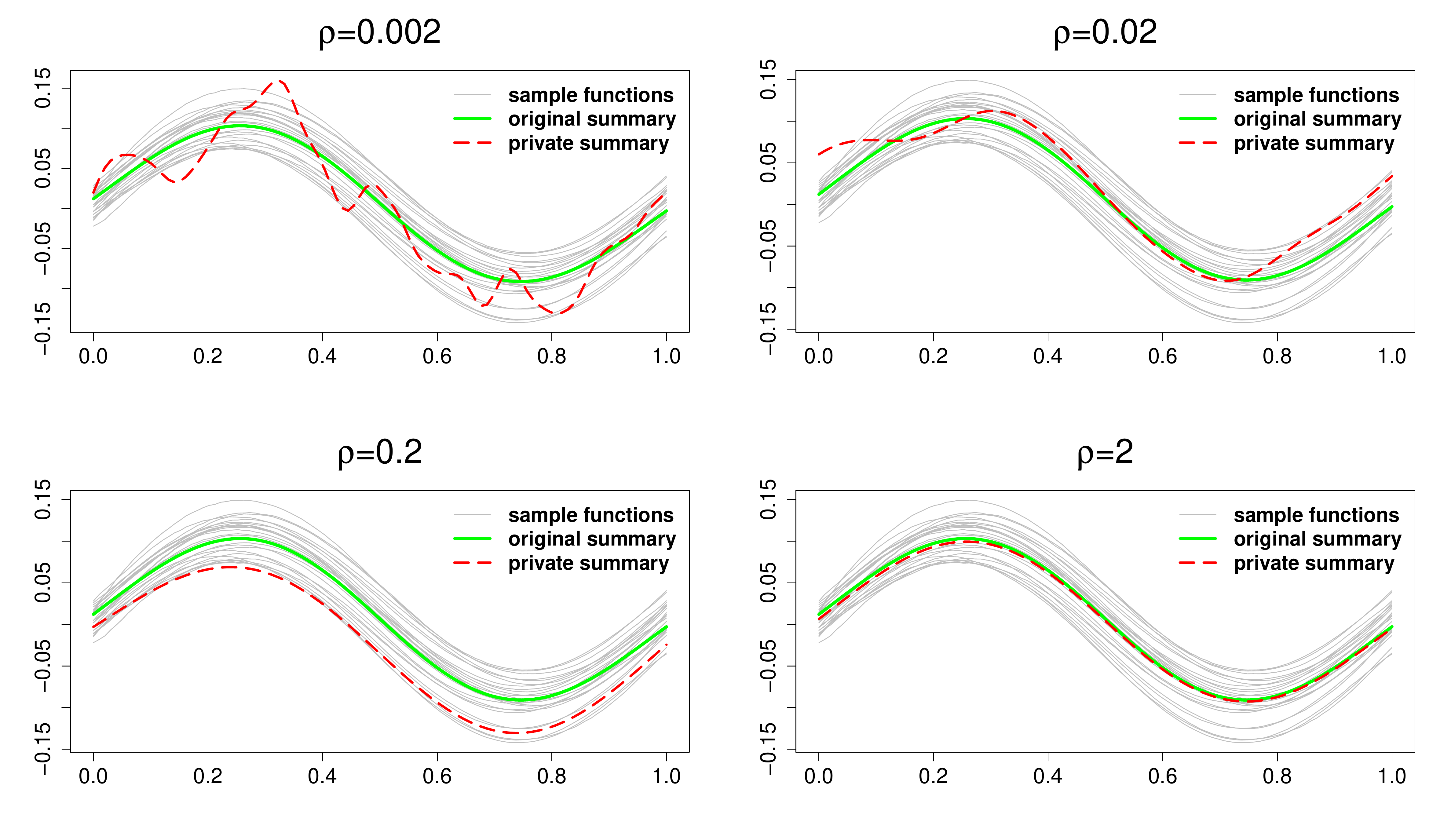}
	\caption{Original and private RKHS smoothing mean with Gaussian Kernel ($C_{1}$) for different values of kernel range parameter $\rho$ }
	\label{DP_ro}
\end{figure}

\subsubsection*{Scenario 3: Varying the kernel function $c(t,s)$}
Here we consider the four different kernels given in \eqref{Kernels} for both the noise and RKHS kernel (which are taken to be the same).  The results are summarized in Figure \ref{DP_kernel}.  All kernels give roughly the same pattern, however, $C_1$ produces curves which are infinitely differentiable, while the exponential kernel produces curves that are nowhere differentiable (they follow an Ornstein-Uhlenbeck process).  The two Mat\'ern covariances give paths that have either one ($C_3$) or two ($C_2$) derivatives.  Since the underlying function to be estimated is already very smooth, the kernel does not have a substantial impact.  However, for more irregular shapes, this choice can play a substantial role on the efficiency of the resulting RKHS estimate.
\begin{figure}[ht!]
	\centering
	\includegraphics[width=0.85\columnwidth]{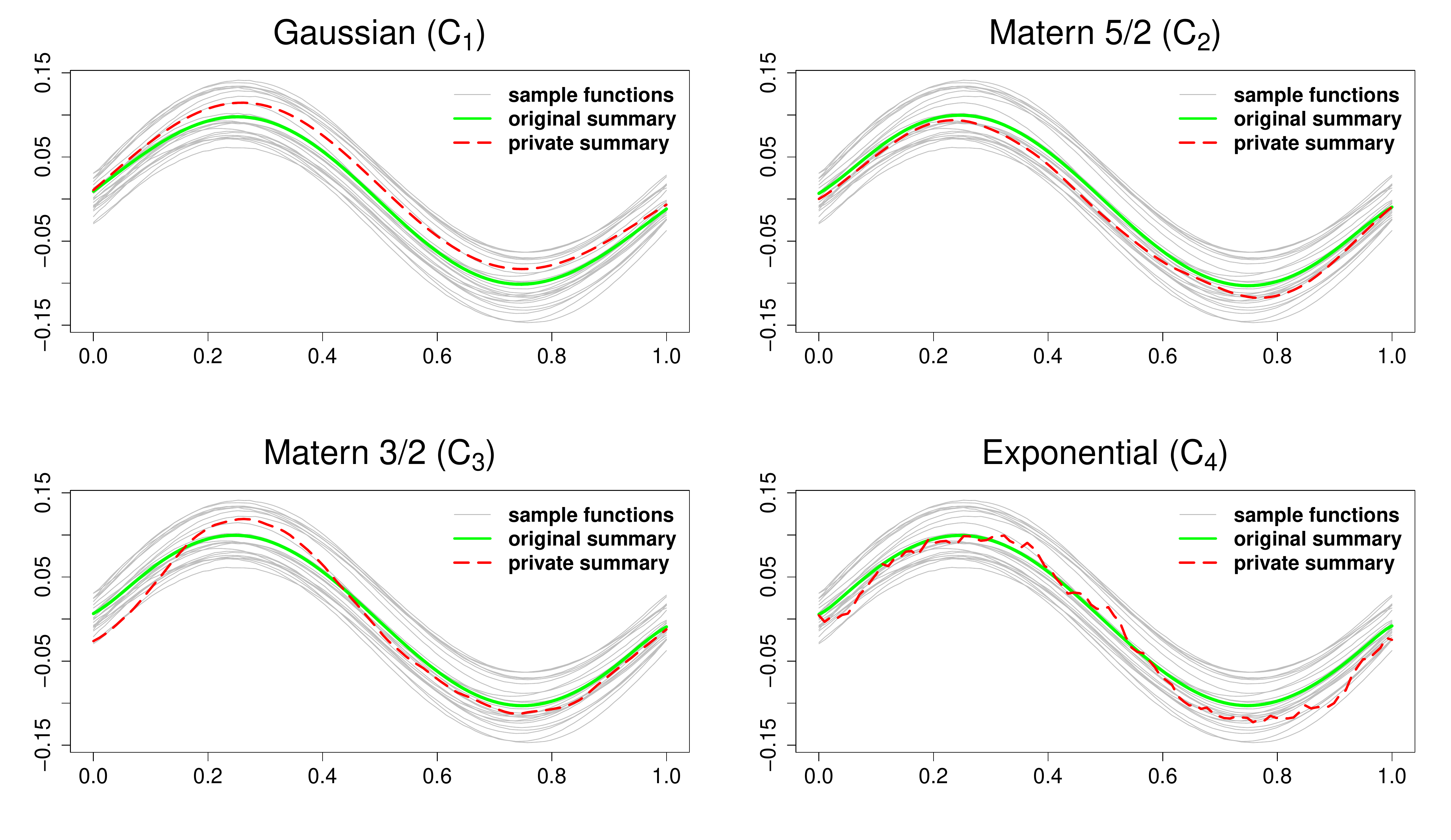}
	\caption{Original and private RKHS smoothing mean for different kernels}
	\label{DP_kernel}
\end{figure}

\subsubsection*{Scenario 4: Varying the smoothing parameter of samples $p$}
In this setting we vary $p$ from $1.1$ to $4$, which determines the smoothness of the data, $X_n(t)$.  Note that $p$ has to be strictly larger than $1$ or the $X_i$ will not be square integrable. Figure \ref{DP_p} summarizes these results.  As we can see, the smoothness of the curves has a smaller impact on the utility of the sanitized estimates as compared to other parameters.  As the curves become smoother, the global sensitivity decreases implying the need for less noise being added in order to maintain the desired privacy level, and thus resulting in a higher utility for the privacy enhanced curves.  However, the smoothness, in terms of derivatives, of the estimates is not affected, as this is determined by the kernel.
\begin{figure}[ht!]
	\centering
	\includegraphics[width=0.85\columnwidth]{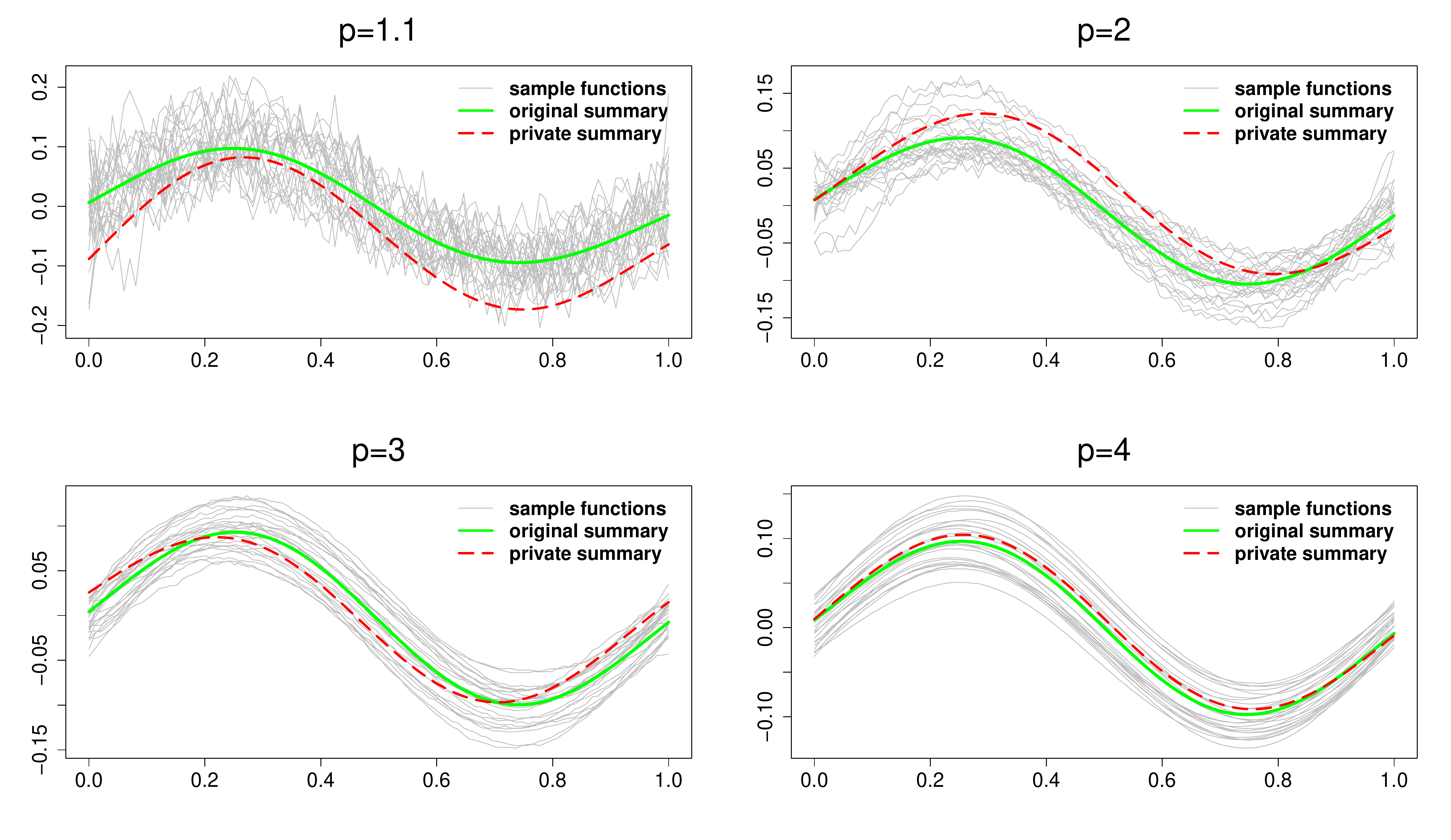}
	\caption{Original and private RKHS smoothing mean with Gaussian Kernel ($C_{1}$) for different values of smoothing parameter of samples $p$}
	\label{DP_p}
\end{figure}

\subsubsection*{Scenario 5: Varying the privacy parameters $(\epsilon,\delta)$}
In this setting we vary the privacy parameters, $\epsilon$ and $\delta$.  Figure \ref{DP_alpha} present the effects of varying $\epsilon$ from $5$ to $0.1$ while in Figure \ref{DP_beta} we vary $\delta$ from $0.1$ to $10^{-6}$. As we decrease the parameters, we are requiring a stricter form of privacy, which is reflected in the plots; recall that $\delta=0$ will give us the stricter form of DP, $\epsilon$-DP (also called $\epsilon$-DP).  As we decrease these values, the overall noise added increases, and we expect larger deviations of the sanitized estimates from the mean.  There is less sensitivity in the output to changes in $\delta$ than to $\epsilon$. However, as with the previous scenario these parameters play no role in the overall smoothness, in terms of derivatives of the resulting estimates.

\begin{figure}[ht!]
	\centering
	\includegraphics[width=0.85\columnwidth]{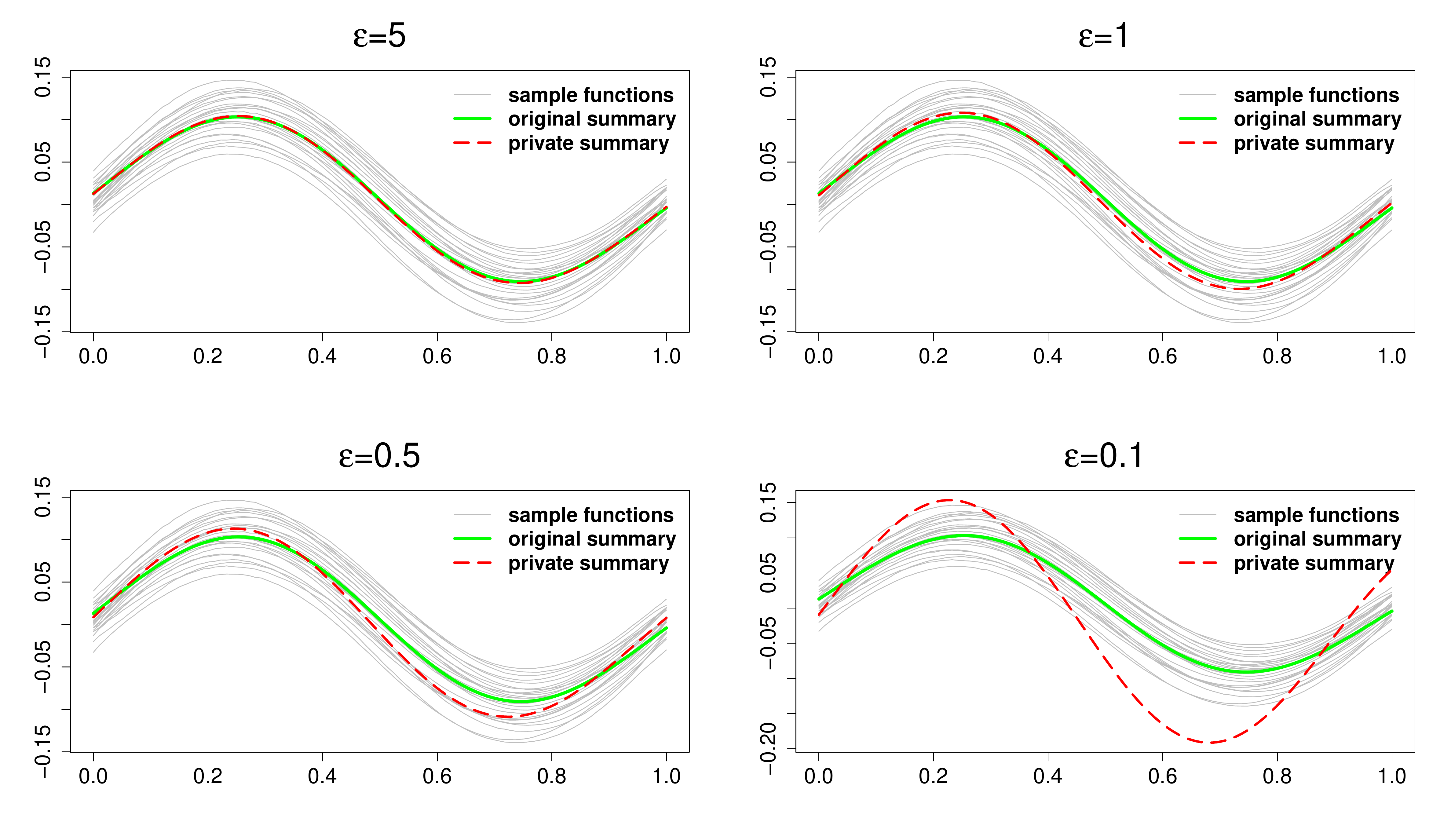}
	\caption{Original and private RKHS smoothing mean with Gaussian Kernel ($C_{1}$) for different values of privacy level parameter $\epsilon$ when $\delta=1$}
	\label{DP_alpha}
\end{figure}

\begin{figure}[ht!]
	\centering
	\includegraphics[width=0.85\columnwidth]{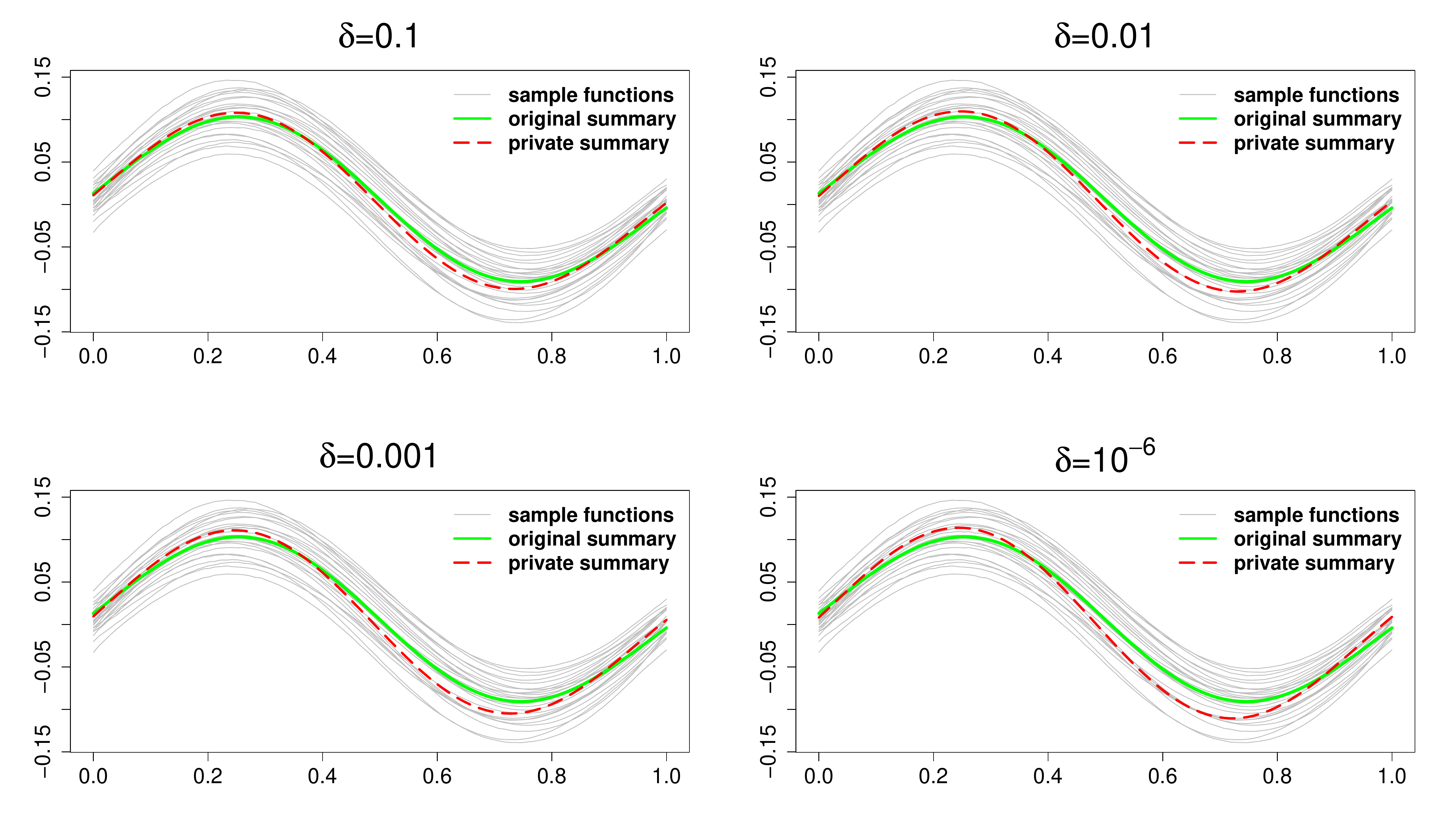}
	\caption{Original and private RKHS smoothing mean with Gaussian Kernel ($C_{1}$) for different values of privacy level parameter $\delta$ when $\epsilon=1$}
	\label{DP_beta}
\end{figure}
\subsubsection*{Scenario 6: Varying sample size N}
In Figure \ref{DP_N} we vary the sample size from $5$ to $100$.  The results are very similar to changing $\delta$ and $\epsilon$, as the sample size does not influence the smoothness of the curves (in terms of derivatives), but the accuracy of the estimate (green curve) gets much better and so does the utility of the privacy enhanced version.
\begin{figure}[ht!]
	\centering
	\includegraphics[width=0.85\columnwidth]{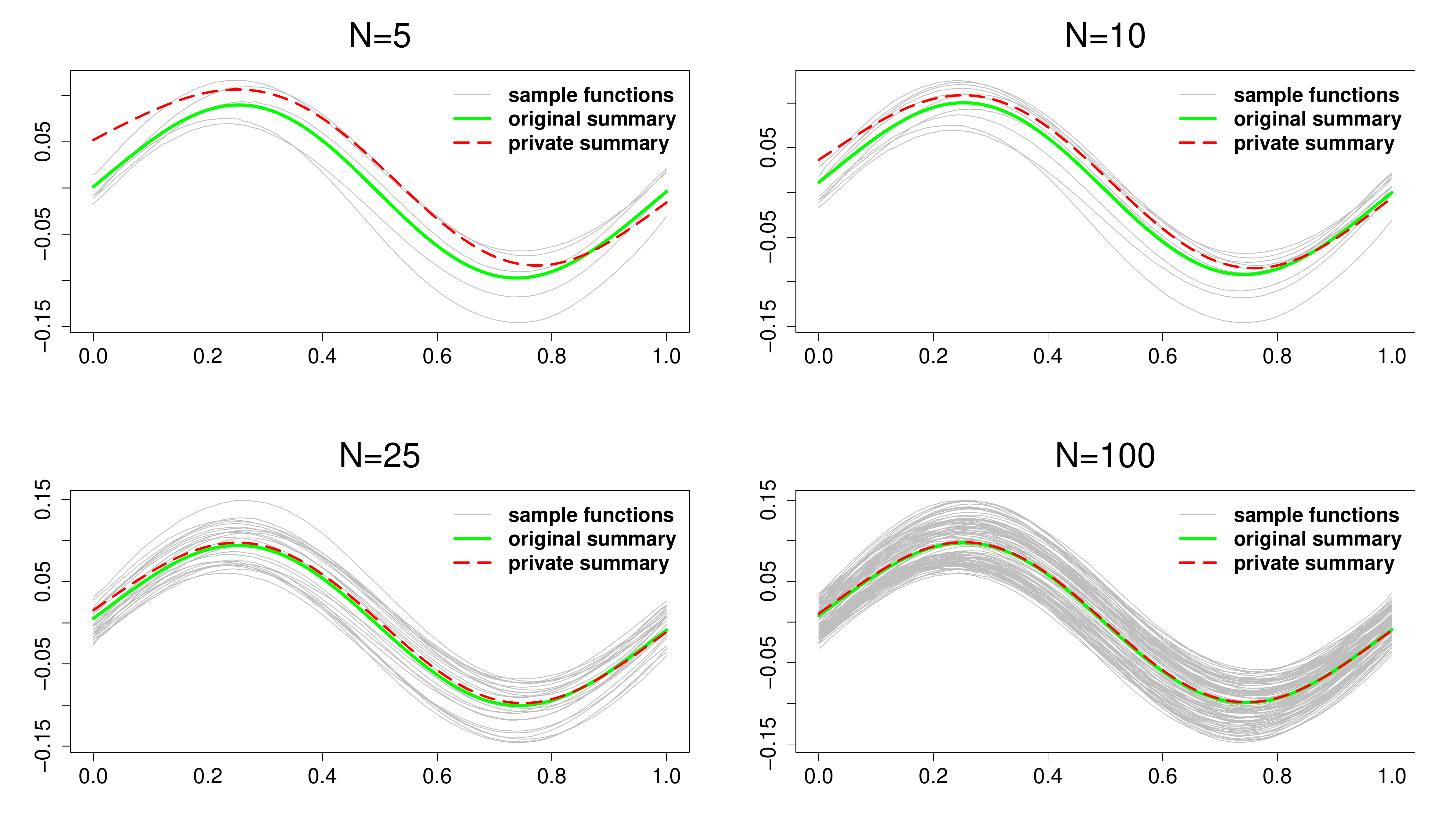}
	\caption{Original and private RKHS smoothing mean with Gaussian Kernel ($C_{1}$) for different sample sizes $N$ }
	\label{DP_N}
\end{figure}
\subsubsection*{Scenario 7: Different underlying mean function $\mu$ }
Lastly, in Figure \ref{DP_mu} we consider three additional mean functions. Overall, the actual function being estimated does not influence the utility of the privacy enhanced version, only the accuracy of the original estimate.  This is because the noise to be added is computed from the different smoothing parameters as well as the range of the $L^2$ norm of the data, not the underlying estimate itself.
%\textcolor{red}{[WHAT DO YOU MEAN HERE BY INITIAL ESTIMATE? AND WHAT DO YOU MEAN By UTILITY HERE? IN GENERAL DO WE HAVE A WAY of measuring UTILITY, did we define what we mean by utility anywhere? closeness of the mean curve and privacy enhanced curve? can we quantify this?]}
\begin{figure}[ht!]
	\centering
	\includegraphics[width=0.85\columnwidth]{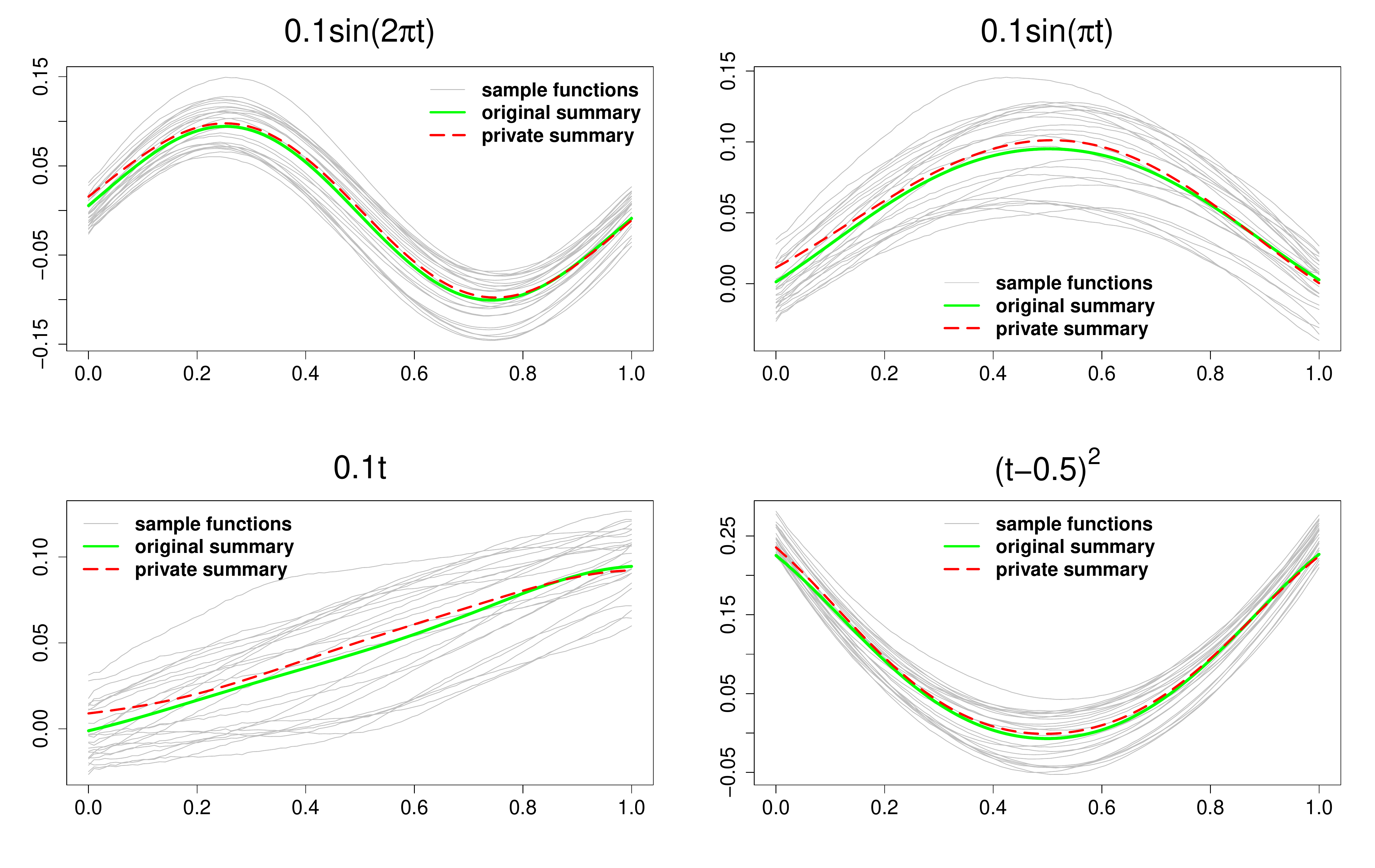}
	\caption{Original and private RKHS smoothing mean with Gaussian Kernel ($C_{1}$) for different initial mean functions $\mu$ }
	\label{DP_mu}
\end{figure}
%%%%%%%%%%%%%%%%%%%%%%%%%%%%%%%%%

%\section{Extension of Applications}
%In this section we illustrate our method on two applications.  The first, in Section \ref{s:eapp}, is an extension of application \ref{s:app} involving brain scans ({\em diffusion tensor imaging}, DTI) of patients with multiple sclerosis for other kernel functions.

%Our second application, in Section \ref{s:face}, stems from collaborative work with Dr.\ Mark Shriver in the department of anthropology at Penn State University.  The study, called {\it ADAPT}, involves collecting 3D facial scans on thousands of human subjects.  Given how identifiable faces are, this application to further highlight the needs for privacy tools.  We also use this data to emphasize just how hard the problem is when working with objects, such as faces, that humans are very good at visually identifying.  We will see in this setting that, compared to our simulations or DTI, it is much harder to produce seemingly reasonable sanitized estimates, which we hope will help motivate other researchers to also consider this type of privacy preservation and data sharing problems.  % {\color{red}WHAT'S DTI?} %

\section{Extension of Diffusion Tensor Imaging} \label{s:eapp}
In this section our aim is to see the privacy enhanced RKHS smoothing estimate of the mean function discussed in Section \ref{s:app} for $C_{3}$ and $C_{4}$ based on the optimal parameters in Table \ref{Par_CV_Reg1} for CV. The results are given in Figures \ref{DTI_M3_R} and \ref{DTI_Exp_R} for Matern and Exponential kernels, respectively. In each case, we see that the utility of the privacy enhanced versions increases as $\phi$ increases, however, the largest values of $\phi$ produce estimates that are over smoothed. Here Table \ref{Par_CV_Irr1} represents the optimal parameters to generate privacy enhanced estimates for PCV.

\begin{table}[ht!]
	\centering
	\begin{tabular}{ccccc}
		\hline
		& & Exp. Kernel & Mat $3/2$ Kernel &Gau. Kernel \\ \hline
		& $\phi$ & optimum $\rho$ & optimum $\rho$ & optimum $\rho$ \\
		\hline
		1 & 0.0001 & 0.25 & 0.10 & 0.01 \\
		2 & 0.001 & 0.20 & 0.15 & 0.01 \\
		3 & 0.01 & 0.30 & 0.15 & 0.03 \\
		4 & 0.03 & 0.80 & 0.30 & 0.05 \\
		\hline
	\end{tabular}
	\caption{Optimum kernel range parameters $\rho$ for different kernels with using CV for each fixed penalty parameter $\phi$ in DTI dataset }  \label{Par_CV_Reg1}
\end{table}
\begin{table}[ht!]
	\centering
	\begin{tabular}{ccccc}
		\hline
		Kernel & range $\phi$ & range $\rho$ & optimum $\phi$ & optimum $\rho$ \\
		\hline
		$C_{1}$ & $[10^{-4},0.1]$ &  $[0.01,0.1]$ & 0.005 & 0.030 \\
		$C_{3}$ & $[10^{-4},0.1]$ &  $[0.05,0.5]$ & 0.005 & 0.250 \\		
		$C_{4}$ & $[10^{-4},0.1]$ & $[0.2,1]$ &0.010 & 0.466 \\
		\hline
	\end{tabular}
	\caption{Optimum penalty and range parameters $(\phi,\rho)$ for different kernels with PCV in CCA application.}  \label{Par_CV_Irr1}
\end{table}
\begin{figure}[ht!]
	\centering
	\includegraphics[width=0.85\columnwidth]{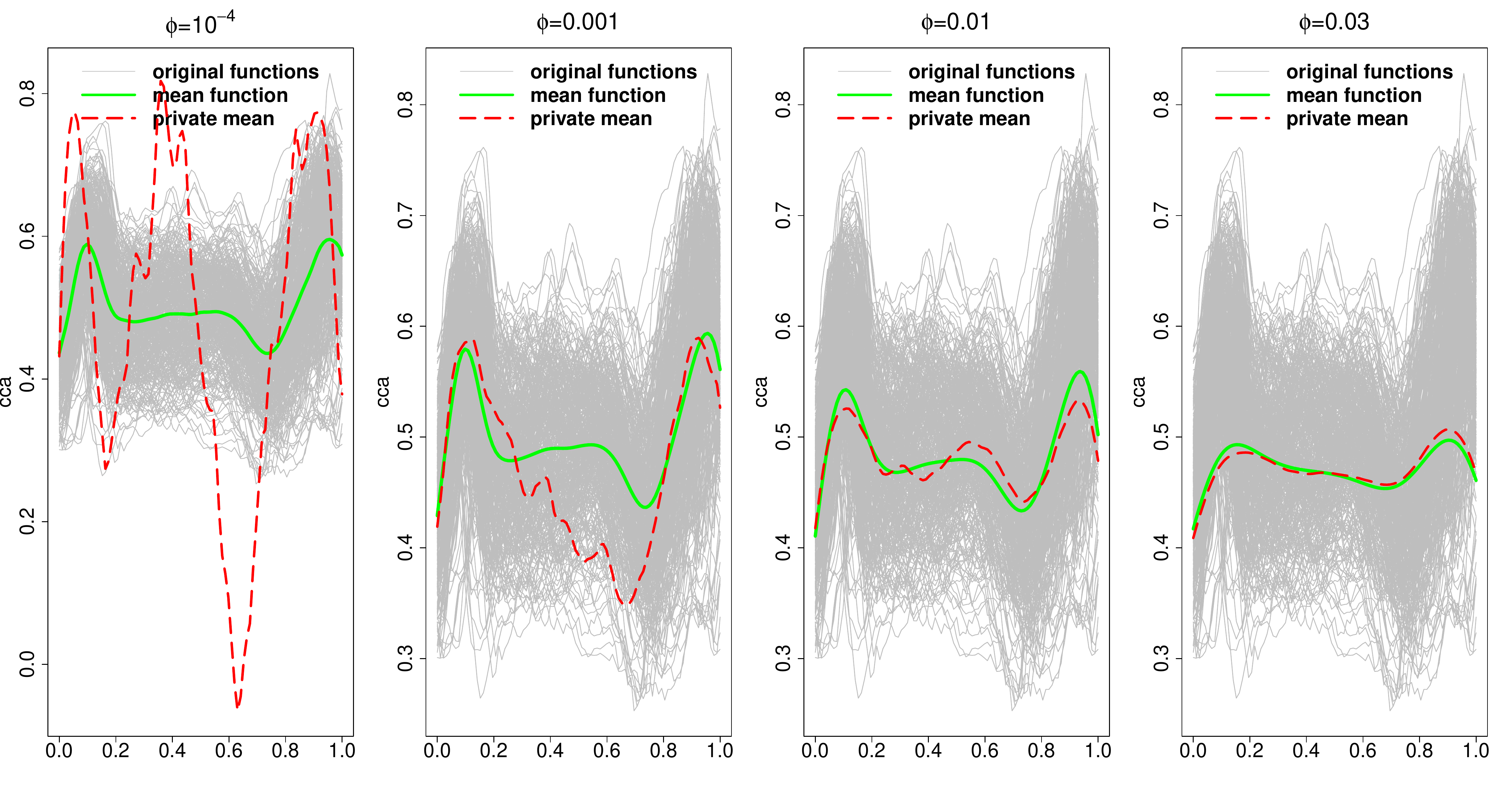}
	\caption{Mean estimate for CCA and its private release using Mat\'ern$(3/2)$ kernel ($C_{3}$) with CV.}
	\label{DTI_M3_R}
\end{figure}
\begin{figure}[ht!]
	\centering
	\includegraphics[width=0.85\columnwidth]{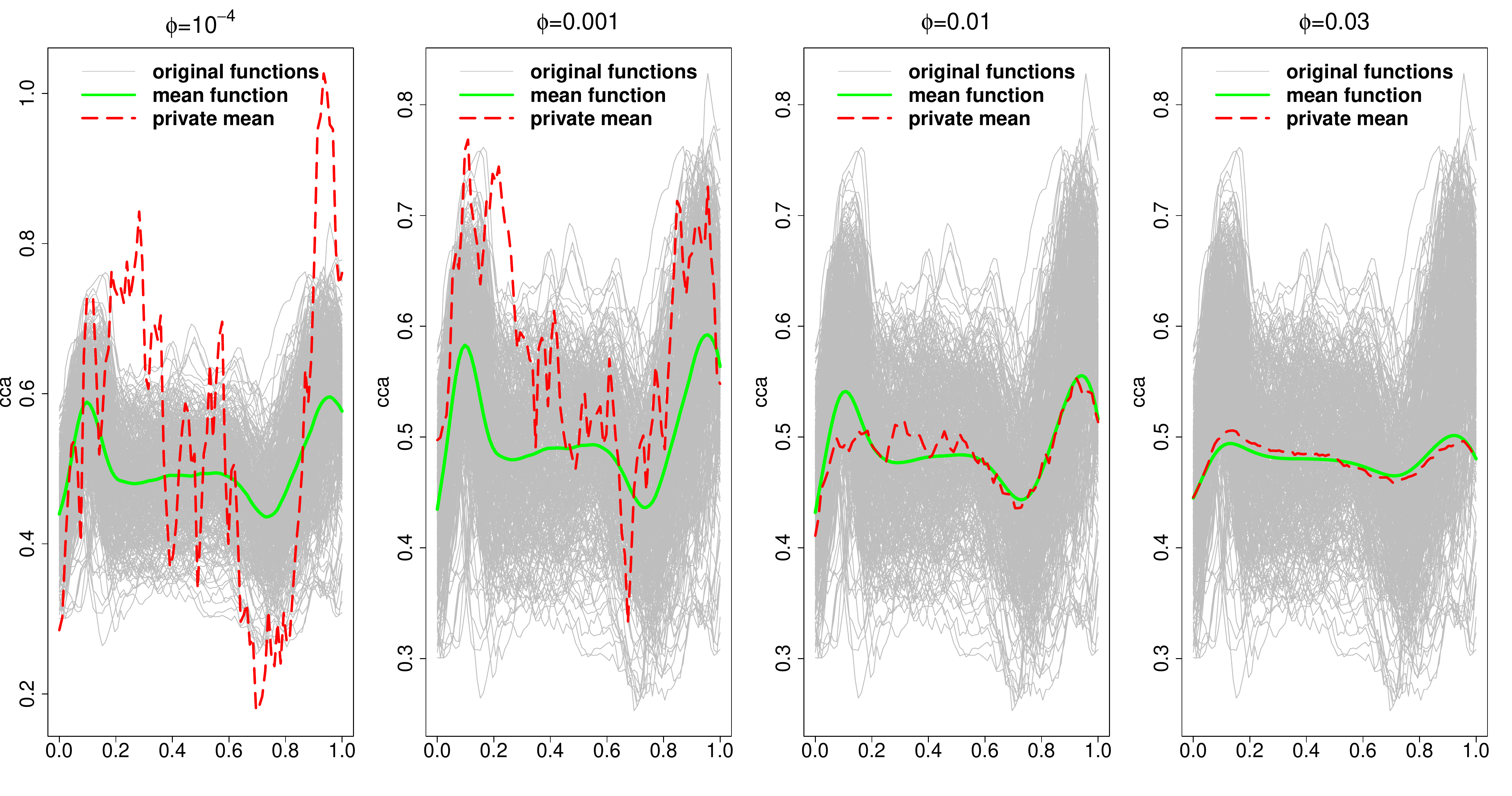}
	\caption{Mean estimate for cca and its private release using Exponential kernel ($C_{4}$) with CV.}
	\label{DTI_Exp_R}
\end{figure}
\end{document}